\documentclass{article}
\pdfoutput=1

\usepackage{graphicx,url,multirow}

\usepackage[utf8]{inputenc}
\usepackage{amsmath}
\usepackage{amssymb}
\usepackage{amsthm}
\usepackage[boxed]{algorithm}
\makeatletter
\renewcommand\fs@boxed{\def\@fs@cfont{\bfseries}\let\@fs@capt\floatc@plain
  \def\@fs@pre{\setbox\@currbox\vbox{\hbadness10000
    \vbox{\hrule\hbox to\hsize{\vrule\kern1pt\vbox{\kern3pt\box\@currbox\kern2pt}\kern-2pt\vrule}\hrule}}}
  \def\@fs@mid{\kern2pt}
  \def\@fs@post{}\let\@fs@iftopcapt\iffalse}
\makeatother
\usepackage[noend]{algorithmic}
\usepackage[all]{xy}
\usepackage{rotating}
\usepackage{enumerate}

\usepackage{color}
\usepackage{multicol}

\usepackage{natbib,url}

\usepackage{rotating}

\theoremstyle{definition}
\newtheorem{theorem}{Theorem}
\newtheorem{definition}[theorem]{Definition}
\newtheorem{example}[theorem]{Example}
\newtheorem{lemma}[theorem]{Lemma}
\newtheorem{corollary}[theorem]{Corollary}
\newtheorem{proposition}[theorem]{Proposition}

\newcommand{\N}{\mathbb{N}}
\newcommand{\Q}{\mathbb{Q}}
\newcommand{\UPTO}{\textbf{to~}}

 \newcommand{\greg}[1]{{#1}}
 \newcommand{\gregtwo}[1]{{#1}}
 \newcommand{\ADDITIONCOLOR}{}
 \newcommand{\jgd}[1]{{#1}}
 \newcommand{\jgdhol}[1]{{#1}}
 \newcommand{\jgdeig}[1]{{#1}}

\newcommand{\eqdef}{\stackrel{\text{def}}{=}}

\newcommand{\spatt}{{\sc SPatt}}
\newcommand{\linbox}{{\sc LinBox}}
\newcommand{\givaro}{{\sc Givaro}}
\newcommand{\mpfr}{{\sc Mpfr}}

\newcommand{\mytitle}{Sparse approaches for the exact distribution of patterns
  in long state sequences generated by a Markov source}
\newcommand{\nuelmail}{\href{mailto:Gregory.Nuel@ParisDescartes.fr}{Gregory.Nuel@ParisDescartes.fr}}
\newcommand{\dumasmail}{\href{mailto:Jean-Guillaume.Dumas@imag.fr}{Jean-Guillaume.Dumas@imag.fr}}
\title{\mytitle}
\author{Gregory \textsc{Nuel}\footnote{MAP5, UMR CNRS 8145,
    Department of Applied Mathematics, Paris Descartes University,
    France. \nuelmail}
\and Jean-Guillaume \textsc{Dumas}\footnote{ Laboratoire
  Jean Kuntzmann, UMR CNRS 5224, Universit\'e de Grenoble,
  Grenoble, France. \dumasmail}
}
\date{}

\usepackage{color,svgcolor}
 \makeatletter
 \usepackage{hyperref}
\hypersetup{
pdftitle={\mytitle},
pdfauthor={Gregory Nuel and Jean-Guillaume Dumas},
breaklinks=true,
colorlinks=true,
linkcolor=darkgreen,
citecolor=darkred,
urlcolor=darkblue,
}
 \makeatother

\begin{document}

\maketitle

 \begin{abstract}
   We present two novel approaches for the computation of the exact
   distribution of a pattern in a long sequence. \greg{Both approaches take
   into account the sparse structure of the problem and are two-part
   algorithms.}
 \jgd{The first approach relies on a partial recursion after a fast
   computation of the second largest
   eigenvalue of the transition matrix of a Markov chain
   embedding. The second approach uses fast Taylor expansions of an
   exact bivariate rational reconstruction of the distribution. }

   We illustrate the interest of both approaches on a simple
   toy-example and two biological applications: the transcription
   factors of the Human Chromosome 10 and the PROSITE signatures of
   functional motifs in proteins. On these example our methods
   demonstrate their complementarity and their ability to extend the
   domain of feasibility for exact computations in pattern problems to
   a new level.
 \end{abstract}


\section{Introduction}

The distribution of patterns in state random sequences
generated by a Markov source has many applications in a wide range
of fields including (but not limited to) reliability, insurance,
communication systems, pattern matching, or bioinformatics. In the
latter field in particular, the detection of statistically exceptional
DNA or protein patterns (PROSITE signatures, CHI motifs, regulation
patterns, binding sites, etc.) have been very successful allowing both
to confirm known biological signals as well as new discoveries. Here
follows a short selection of such work:
\citet{KBC92,Hel98,BJVU98,KBSG99,BFWCG00,Spouge02,HKB02,Spouge04}.

From the technical point of view, obtaining the distribution of a
pattern count in a state random sequence is a difficult problem
which has drawn a considerable interest from the probabilistic,
combinatorics and computer science community over the last fifty
years. Many concurrent approaches have been suggested, all of them
having their own strengths and weaknesses, and we give here only a short
selection of the corresponding references 
\citep[see][for more comprehensive reviews]{reignier,Lot05,nuel}.

Exact methods are based on a wide range of techniques like Markov
chain embedding, \greg{probability generating functions}, combinatorial methods,
or exponential families
\cite{Fu96,StP97,Ant01,Cha05,boeva,Nue06,StS07,boeva2}. There is also
a wide range of asymptotic approximations, the most popular among them
being: Gaussian approximations \cite{PBM89,Cow91,KlB97,Pru95}, Poisson
approximations \cite{God91,GGSS95,ReS99,Erh00} and Large deviation
approximations \cite{DRV01,Nue04}.

More recently, several authors
\citep{nicodeme,stefanov,lladser,Nue08,ribeca} pointed out the strong
connection of the problem to the theory of pattern matching by
providing the optimal Markov chain embedding of any pattern problem
through minimal Deterministic Finite state Automata (DFA). However,
this approach remains difficult to apply in practice when considering
high complexity patterns and/or long sequences.

In this paper, we want to address this problem by suggesting two
efficient ways to obtain the distribution of any pattern of interest
in a (possibly long) homogeneous state Markov sequence.

The paper is organized as follow. In the first part, we recall (in
Section~\ref{sec:dfa}) the
principles of optimal Markov chain embedding through DFA, as well as
the associated \greg{probability generating function (pgf)} formulas. 
We then (in Section~\ref{sec:partial}) present a new algorithm using
partial recursion formulas. 
The convergence of these partial recursion formulas depends on a (fast)
precomputation of the second largest eigenvalue of the transition
matrix of a Markov chain embedding. 
The next part (Section~\ref{sec:highorder}) takes advantage of
state-of-the-art results in exact computation 
to suggest a very efficient way to obtain the bivariate \greg{pgf} of the
problem through rational reconstruction.
\jgd{Once this involving precomputation has been performed,
fast Taylor expansions, using the high-order liftings of
\cite{Storjohann:2003:jsc},
can very quickly reveal the distribution of any pattern count.}
We then (in Section~\ref{sec:appli}) apply our new algorithms
successively to a simple toy-example, a selection of DNA
(Transcription Factors) patterns, and to protein motifs (PROSITE
signature). In all cases, the relative performance of the two
algorithms are presented and discussed, highlighting their strengths
and weaknesses. We conclude (in Section~\ref{sec:conclusion}) with
some perspectives of this work, including a table that summarizes
memory and time complexities.

\section{DFA and optimal Markov chain embedding}\label{sec:dfa}

\subsection{Automata and languages}\label{sec:automata_languages}

In this part we recall some classical definitions and results
of the well known theory of languages and automata \citep{HMU01}.

We consider $\mathcal{A}$ a \emph{finite alphabet}
whose elements are called \emph{letters}.  A \emph{word} (or
\emph{sequence}) over $\mathcal{A}$ is a sequence of letters and a
\emph{language} over $\mathcal{A}$ is a set of words (finite or
not). We denote by $\varepsilon$ the \emph{empty word}.  For example
${\tt ABBABA}$ is a word over the binary alphabet $\mathcal{A}=\{{\tt
  A},{\tt B}\}$ and $\mathcal{L}=\{{\tt AB},{\tt ABBABA},{\tt
  BBBBB}\}$ is a (finite) language over $\mathcal{A}$.

The \emph{product} $\mathcal L_1 \cdot \mathcal L_2$ (the dot could be
omitted) of two languages is the language $\{w_1w_2, w_1 \in \mathcal
L_1, w_2 \in \mathcal L_2\}$ where $w_1w_2$ is the concatenation -- or
product -- of $w_1$ and $w_2$.  If $\mathcal L$ is a language,
$\mathcal L^n=\{w_1\ldots w_n \text{ with }$ $w_1,\ldots w_n \in
\mathcal L\}$ and the \emph{star closure} of $\mathcal L$ is defined
by $\mathcal L^* = \cup_{n \geqslant 0} \mathcal L^n$.  The language
$\mathcal{A}^*$ is hence the set of all possible words over
$\mathcal{A}$.  For example we have $\{{\tt AB}\} \cdot \{{\tt
  ABBABA},{\tt BBBBB}\}= \{{\tt ABABBABA},{\tt ABBBBBB}\}$; $\{{\tt
  AB}\} ^3 =\{{\tt ABABAB}\}$ and $\{{\tt AB}\}^*=\{\varepsilon,{\tt
  AB},{\tt ABAB},{\tt ABABAB},{\tt ABABABAB}\ldots\}$.

A \emph{regular language} is either the empty word, or a single
letter, or obtained by a finite number of \emph{regular operations}
(namely: union, product and star closure). A finite sequence of
regular operations describing a regular language is called a
\emph{regular expression} whose size is defined as its number of
operations. $\mathcal{A}^*$ is regular.  Any finite language is
regular.

\begin{definition}
  If $\mathcal{A}$ is a {finite alphabet}, $\mathcal{Q}$ is a {finite set of
    states}, $\sigma \in \mathcal Q$ is a {starting state}, $\mathcal F
  \subset \mathcal Q$ is a {subset of final states} and $\delta: \mathcal
  Q \times \mathcal A \rightarrow \mathcal Q$ is a {transition function},
  then $(\mathcal{A},\mathcal{Q},\sigma,\mathcal{F},\delta)$ is a
  \emph{Deterministic Finite Automaton} (DFA).  For all
  $a_1^d=a_1\ldots a_{d-1}a_d \in \mathcal{A}^d$ ($d \geqslant 2$) and
  $q\in \mathcal{Q}$ we recursively define $\delta(q,a_1^d)=\delta(
  \delta(q,a_1^{d-1}) ,a_d)$.  A word $w \in \mathcal{A}^h$
  is \emph{accepted} (or \emph{recognized}) by the DFA if
  $\delta(\sigma,w) \in \mathcal{F}$. The set of all words accepted by a
  DFA is called its language. See in Figure~\ref{fig:dfa_w1} a
  graphical representation of a DFA.
\end{definition}

\begin{figure}
  \begin{center}
    \includegraphics[width=1.0\textwidth,trim=50 0 0 0]{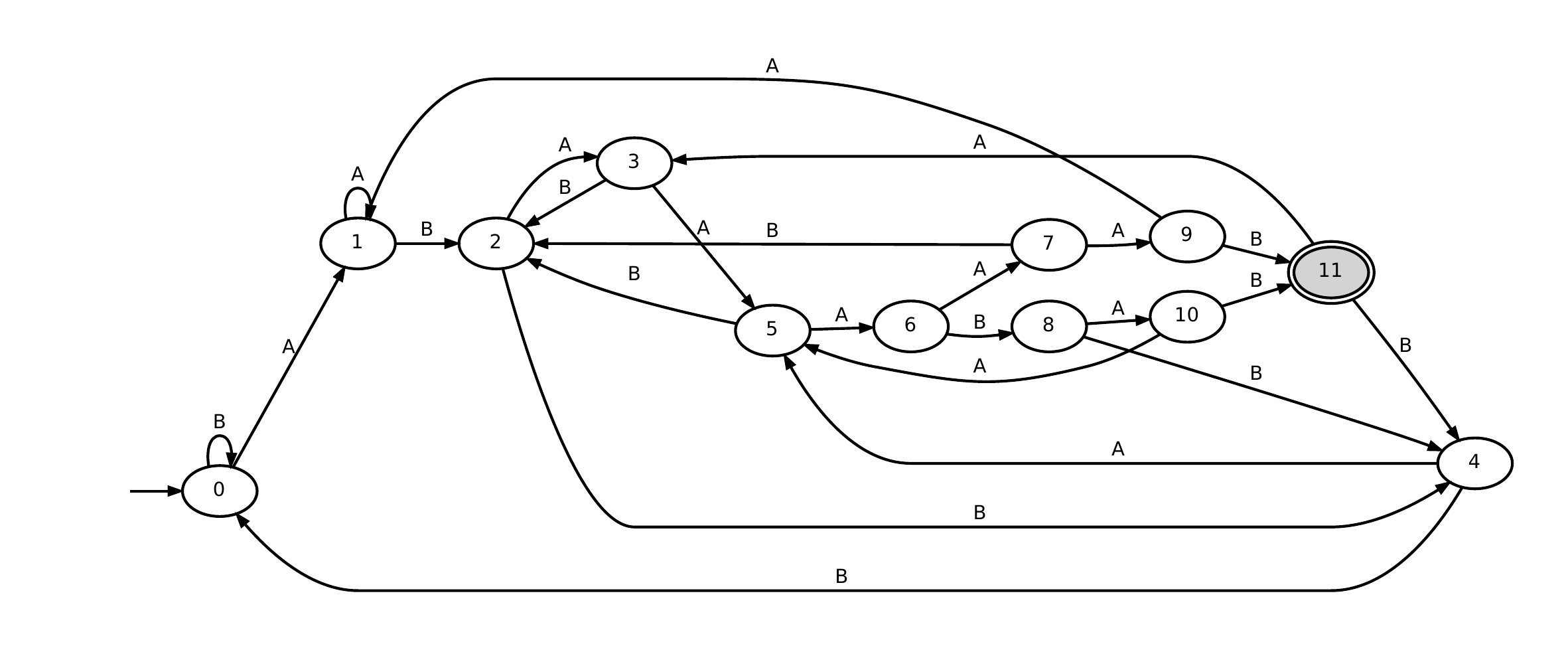}
  \end{center}
  \caption{Graphical representation of the DFA
    $(\mathcal{A},\mathcal{Q},\sigma,\mathcal{F},\delta)$ with
    $\mathcal{A}=\{{\tt A},{\tt B}\}$,
    $\mathcal{Q}=\{0,1,2,\ldots,10,11\}$, $\sigma=0$, $\mathcal{F}=\{11\}$
    and $\delta(0,{\tt A})=1$, $\delta(0,{\tt B})=0$,
    $\delta(1,{\tt A})=1$, $\delta(1,{\tt B})=2$,
    $\delta(2,{\tt A})=3$, $\delta(2,{\tt B})=4$,
    $\delta(3,{\tt A})=5$, $\delta(3,{\tt B})=1$,
    $\delta(4,{\tt A})=5$, $\delta(4,{\tt B})=0$,
    $\delta(5,{\tt A})=6$, $\delta(5,{\tt B})=2$,
    $\delta(6,{\tt A})=7$, $\delta(6,{\tt B})=8$,
    $\delta(7,{\tt A})=9$, $\delta(7,{\tt B})=2$,
    $\delta(8,{\tt A})=10$, $\delta(8,{\tt B})=4$,
    $\delta(9,{\tt A})=1$, $\delta(9,{\tt B})=11$,
    $\delta(10,{\tt A})=5$, $\delta(10,{\tt B})=11$,
    $\delta(11,{\tt A})=3$ and $\delta(11,{\tt B})=4$.  This DFA is
    the smallest one that recognizes the language
    $\mathcal{L}=\mathcal{A}^*\mathcal{W}_1$ with
    $\mathcal{A}=\{{\tt A},{\tt B}\}$, $\mathcal{W}_1 =
    {\tt AB}\mathcal A^1 {\tt AA} \mathcal A^1 {\tt AB}$ and hence
    $|\mathcal{W}_1|=4$.  }\label{fig:dfa_w1}
\end{figure}

We can now give the most important result of this part which is a
simple application of the classical Kleene and Rabin \& Scott theorems
\citep{HMU01}:
\begin{theorem}\label{thm:smallest_DFA}
  For any regular language $\mathcal{L}$ there exists a unique (up to
  a unique isomorphism) smallest DFA whose language is
  $\mathcal{L}$. If we denote by $E$ the size of the regular
  expression corresponding to $\mathcal{L}$, then the size $R$ of the
  smallest DFA is bounded by $2^E$.
\end{theorem}

\greg{
For certain specific patterns (e.g. $\mathcal{A}^*w$ where $w$ is a
simple word), a minimal DFA can be built directly using ad hoc
construction or well-known algorithms \citep[e.g. Aho-Corasick
algorithm,][]{aho1975efficient}. In general however, building a
minimal DFA from a regular expression usually requires three steps:
\begin{enumerate}[1)]
\item turning the regular expression into a
Nondeterministic Finite Automaton -- NFA -- \citep[Thompson's or Glushkov's algorithm,][]{allauzen2006unified};
\item producing a DFA from the NFA \citep[determinization; Subset construction, see][Section 2.3.5]{HMU01};
\item minimizing the DFA
by removing unnecessary states \citep[for minimization; Hopcroft's algorithm, see][Section 4.4.3]{hopcroft1971n,HMU01}.
\end{enumerate}
}
For instance, The
\spatt\footnote{\url{http://www.mi.parisdescartes.fr/~nuel/spatt}}
software allows to compute these DFA from regular expressions.

Now, as stated in Theorem~\ref{thm:smallest_DFA}, the smallest DFA may have
a total of $2^E$ states in the worse case. However, this upper bound
is seldom reached in practice. This may be observed in Table
\ref{tab:Wk} where the practical value of $R$ is far below the upper
bound.

\begin{table}
  {\footnotesize \setlength\arraycolsep{2pt}
  $$
  \begin{array}{|c|cccccccccc|}
    \hline
    k & 1 & 2 & 3 & 4 & 5 & 6 & 7 & 8 & 9 & 10 \\
    \hline
    |\mathcal{W}_k| & 4 & 16 & 64 & 256 & 1\,024  & 4\,096 & 16\,384 & 65\,536 & 262\,144 & 1\,048\,576 \\

    2^E & 512 &    2\,048 &     8\,192 &    32\,768&    1.3 \times 10^5&    5.2\times 10^5 &  2.1 \times 10^6
   & 8.4 \times 10^6 & 3.4 \times 10^7 & 1.3 \times 10^8 \\ 
    R & 12 & 27 & 57 & 122 & 262 & 562 &  1\,207 & 2\,592 & 5\,567 & 11\,957 \\
    F & 1 & 3 & 6 & 13 & 28 & 60 & 129 & 277 & 595 & 1\,278 \\
    \hline
  \end{array}
  $$
}
  \caption{Characteristics of the smallest DFA that recognizes the language $\mathcal{L}=\mathcal{A}^* \mathcal{W}_k$ with $\mathcal{A}=\{{\tt A},{\tt B}\}$ and $\mathcal{W}_k = {\tt AB}\mathcal A^k  {\tt AA} \mathcal A^k {\tt AB}$. The pattern cardinality is $|\mathcal{W}_k|=2^k \times 2^k = 4^k$, $R$ is the total number of states, $F$ the number of final states, and $2^E=2^{7+2k}$ is the theoretical upper bound of $R$.}\label{tab:Wk}
\end{table}

\greg{One should note that the complexity $R$ of real-life patterns is
  quite different from one problem to another. For example in
  \citet{gregory2010exact}, the authors consider a total of  $1,276$
  protein signatures from the PROSITE database, for which complexities
  range from $R=22$ ({\tt RGD} motif) to $R=837,507$ (APPLE motif,
  PS00495), with a mode around $R=100$.}

\subsection{Connection with patterns}\label{sec:connection_pattern}

We call \emph{pattern} (or \emph{motif}) over the finite alphabet
$\mathcal{A}$ any regular language (finite or not) over the same
alphabet.  For any pattern $\mathcal{W}$ any DFA that recognizes the
regular language $\mathcal{A}^* \mathcal{W}$ is said to be
\emph{associated} with $\mathcal{W}$. According to Theorem
\ref{thm:smallest_DFA}, there exists a unique (up to unique isomorphism)
smallest DFA associated with a given pattern. For example, if we work
with the binary alphabet $\mathcal{A}=\{{\tt A},{\tt B}\}$, then the
smallest DFA associated with Pattern $\mathcal{W}_1={\tt
  AB}\mathcal A^1 {\tt AA} \mathcal A^1 {\tt AB}$ has $R=12$ states
and $F=1$ final state (see Figure~\ref{fig:dfa_w1}), and the smallest
DFA associated  with Pattern $\mathcal{W}_2={\tt
  AB}\mathcal A^2 {\tt AA} \mathcal A^2 {\tt AB}$ has $R=27$ states
and $F=3$ final states (see Figure~\ref{fig:dfa_w2}).

\begin{figure}
  \begin{center}
    \includegraphics[width=\textwidth,trim=50 0 0 0]{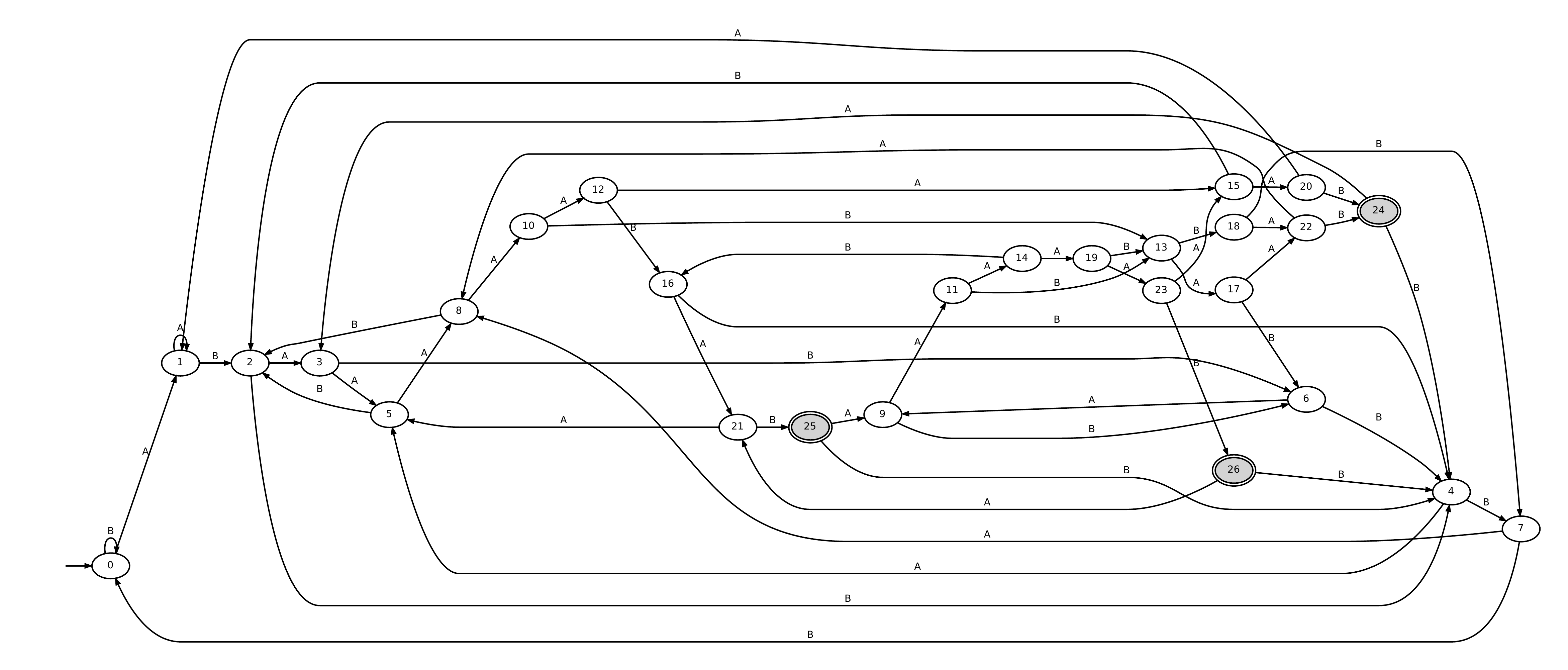}
  \end{center}
  \caption{ Graphical representation of the smallest DFA associated
    with $\mathcal{W}_2={\tt AB}\mathcal{A}^2{\tt AA}\mathcal{A}^2{\tt
      AB}$ with $\mathcal{A}=\{{\tt A},{\tt B}\}$.  This DFA has
    $R=27$ states including $F=3$ final states.}\label{fig:dfa_w2}
\end{figure}

It is well known from the pattern matching theory \citep{CLR90,CH97}
that such a DFA provides a simple way to find all occurrences of the
corresponding pattern in a sequence.

\begin{proposition}\label{prop:dfa_count}
  Let $\mathcal{W}$ be a pattern over the finite alphabet
  $\mathcal{A}$ and
  $(\mathcal{A}$, $\mathcal{Q}$, $\sigma$, $\mathcal{F}$, $\delta)$ be a DFA that is
  associated to $\mathcal{W}$. For each deterministic sequence
  $x_1^\ell=x_1 x_2 \ldots x_\ell$ over $\mathcal{A}$, we recursively
  define the sequence $y_0^\ell=y_0 y_1 \ldots y_\ell$ over
  $\mathcal{Q}$ by $y_0=\sigma$ and $y_i=\delta(y_{i-1},x_i)$. For all
  $1 \leqslant i \leqslant \ell$ we then have the following
  property\footnote{$x_1^i \in \mathcal{A}^* \mathcal{W}$ means that
    an occurrence of $\mathcal{W}$ ends in position $i$ in
    $x_1^\ell$.}: $x_1^i \in \mathcal{A}^* \mathcal{W} \iff y_i \in
  \mathcal{F}$.
\end{proposition}
\begin{proof}
  Since the DFA is associated to $\mathcal{W}$, $x_1^i \in
  \mathcal{A}^* \mathcal{W}$ is equivalent to $\delta(\sigma,x_1^i)
  \in \mathcal{F}$. One can then easily show by induction that
  $\delta(\sigma,x_1^i)=y_i$ and the proof is achieved.
\end{proof}

\begin{example}\label{ex:1}
  Let us consider the pattern $\mathcal{W}_1 = {\tt AB}\mathcal A^1
  {\tt AA} \mathcal A^1 {\tt AB}$ over the binary alphabet
  $\mathcal{A}=\{{\tt A},{\tt B}\}$. Its smallest associated DFA is
  represented in Figure~\ref{fig:dfa_w1}. If $x_1^{20}={\tt
    ABAAABBAAAABBAABABAB}$ is a binary sequence, we build the sequence
  $y_0^{20}$ according to Proposition~\ref{prop:dfa_count} and we get:
{\setlength{\arraycolsep}{4.5pt}
$$
  \begin{array}{*{22}c}
    x_1^{20} = & - & {\tt A} & {\tt B} & {\tt A} & {\tt A} & {\tt A} 
    & {\tt B} & {\tt B} & {\tt A} & {\tt A} & {\tt A} 
    & {\tt A} & {\tt B} & {\tt B} & {\tt A} & {\tt A}
    & {\tt B} & {\tt A} & {\tt B} & {\tt A} & {\tt B}\\
    \hline
    y_0^{20} = & 0 &1&2&3&5&6&8&4&5&6&7&9&{\bf 11}&4&5&6&8&10&{\bf 11}&3&2
  \end{array}
  $$}
  where final states are in bold.  We hence see two occurrences of
  $\mathcal{W}_1$: one ending in position 12 (${\tt ABBAAAAB}$) and
  one in position 18 (${\tt ABBAABAB}$, which overlaps the previous
  occurrence).
\end{example}

If this approach may be useful to localize pattern occurrences in
deterministic sequences, one should note that NFA (Nondeterministic
Finite Automata) are usually preferred over DFA for such a task since
they are far more memory efficient and can achieve similar speed
thanks to lazy determinization \citep{le2011regular}. The DFA-based approach
however has a great advantage when we work with random sequences.

\subsection{Markov chain embedding}

\greg{Let $X_1^\ell$ be a homogeneous\footnote{\greg{Please note that Theorem~\ref{thm:dfa} can be easily generalized to heterogeneous Markov chains, but we focus here on the simpler case since our computational approaches are only valid for homogeneous Markov chains.}} order $m \geqslant 1$ Markov chain over $\mathcal{A}$ whose starting distribution and transition matrix are given for all $(a,b) \in \mathcal{A}^m \times \mathcal{A}$  by $\mu(a)\eqdef\mathbb{P}(X_1^m=a)$ and $\pi(a,b)\eqdef\mathbb{P}(X_i=b | X_{i-m}^{i-1}=a)$. Let now $\mathcal{W}$ be a regular expression over $\mathcal{A}$, our aim being to obtain the distribution of the random number of occurrences of $\mathcal{W}$ in $X_1^\ell$ defined\footnote{\greg{For simplification, we deliberately ignore possible occurrences of $\mathcal{W}$ in $X_1^ m$.}} by:
\begin{equation}
N_\ell\eqdef\sum_{i=m+1}^{\ell} \mathbf{1}_{\{X_1^i \in \mathcal{A}^*\mathcal{W}\}}
\end{equation}
where $\mathbf{1}_\mathcal{E}$ is the indicatrix function of event $\mathcal{E}$ (the function takes the value $1$ is $\mathcal{E}$ is true, $0$ else).}

Let $(\mathcal{A},\mathcal{Q},\sigma,\mathcal{F},\delta)$ be a
\emph{minimal} DFA associated to $\mathcal{W}$. We additionally assume
that this automaton is non $m$-ambiguous \citep[a DFA having this property
is also called an $m$-th order DFA in][]{lladser} which means that
for all $q \in \mathcal{Q}$,
$\delta^{-m}(p)\stackrel{\text{def}}{=}\left\{a_1^m \in
  \mathcal{A}_1^m, \exists p \in \mathcal{Q},
  \delta\left(p,a_1^m\right)=q\right\}$ is either a singleton, or the
empty set. When the notation is not ambiguous, $\delta^{-m}(p)$ may
also denote its unique element (singleton case). We then have the
following result:
\begin{theorem}\label{thm:dfa}
	  We consider the random sequence over $\mathcal{Q}$ defined by
  $Y_0\stackrel{\text{def}}{=}\sigma$ and
  $Y_i\stackrel{\text{def}}{=}\delta(Y_{i-1},X_i)$,
  $\forall i,1 \leqslant i \leqslant \ell$. Then $(Y_i)_{i
    \geqslant m}$ is a homogeneous order 1 Markov chain over
  $\mathcal{Q}'\stackrel{\text{def}}{=}\delta(\greg{ \sigma},\mathcal{A}^m\mathcal{A}^*)$
   such that, for all
  $p,q \in \mathcal{Q}'$ and $1 \leqslant i\leqslant \ell-m$, the
  starting vector $\greg{\mathbf{u}_p}\stackrel{\text{def}}{=}\mathbb{P}\left(
    Y_m=p \right)$ and the transition matrix
  $\greg{\mathbf{T}_{p,q}}\stackrel{\text{def}}{=}\mathbb{P}\left(
    Y_{i+m}=q | Y_{i+m-1}=p \right)$ are given
  by:
\begin{equation}
\greg{\mathbf{u}_p}=\left\{
\begin{array}{ll}
  \mu\left(\delta^{-m}(p) \right) & \text{if $\delta^{-m}(p)\neq\emptyset$}\\
  0 & \text{otherwise}\\
\end{array}
\right.;
\end{equation}
\begin{equation}
\greg{\mathbf{T}_{p,q}}=\left\{
\begin{array}{ll}
  \pi\left(\delta^{-m}(p),b\right)
  & \text{if $\exists b \in \mathcal{A},\delta(p,b)=q$}\\
  0 & \text{otherwise}
\end{array}
\right.
\end{equation}
and we have the following
  property:
\begin{equation}
X_1^i \in \mathcal{A}^* \mathcal{W} \iff Y_i \in
  \mathcal{F}.
\end{equation}
\end{theorem}
\begin{proof}
  The result is immediate considering the properties of the DFA and Proposition~\ref{prop:dfa_count}. See
  \cite{lladser} or \cite{Nue08} for more details.
\end{proof}

\greg{From now on, we consider the decomposition $\mathbf{T}=\mathbf{P}+\mathbf{Q}$ where for all $p,q \in \mathcal{Q}'$ we have:
\begin{equation}
\mathbf{P}_{p,q}=\left\{
\begin{array}{ll}
\mathbf{T}_{p,q} & \text{if $q \notin \mathcal{F}$} \\
0 & \text{if $q \in \mathcal{F}$} \\
\end{array}
\right.
\quad\text{and}\quad
\mathbf{Q}_{p,q}=\left\{
\begin{array}{ll}
0 & \text{if $q \notin \mathcal{F}$} \\
\mathbf{T}_{p,q} & \text{if $q \in \mathcal{F}$} \\
\end{array}
\right..
\end{equation}
We finally introduce the dimension $R\eqdef|\mathcal{Q}'|$ and the
column vector $\mathbf{v}\eqdef (1, \ldots, 1)^\top$ of size $R$, where $^\top$ denotes the transpose symbol.}

\begin{corollary}\label{cor:s}
  With the same hypothesis and notations as in Theorem~\ref{thm:dfa}, the \greg{probability generating function (pgf)} of $N_\ell$ is
  then explicitly given by:
  \begin{equation}\label{eq:mgfGl}
    G_\ell(y)\eqdef\sum_{n \geqslant 0} \mathbb{P}(N_\ell = n) y^n = \mathbf{u} (\mathbf{P} +y\mathbf{Q})^{\ell-m} \mathbf{v}
  \end{equation}
and we also have:
   \begin{equation}\label{eq:mgfG}
    G(y,z)\eqdef\sum_{\ell \geqslant m} \sum_{n \geqslant 0} \mathbb{P}(N_\ell=n) y^n z^\ell=\sum_{\ell \geqslant m} G_\ell(y) z^{\ell} = \mathbf{u} (\mathbf{I} - z\mathbf{P} -yz\mathbf{Q})^{-1} \mathbf{v} z^m
  \end{equation}
  \greg{where $\mathbf{I}$ denotes the identity matrix.}
\end{corollary}
\begin{proof}
  It is clear that $\mathbf{u}\mathbf{T}^\ell$ gives the marginal
  distribution of $Y_\ell$. If we now introduce the dummy variable $y$
  to keep track of the number of pattern occurrences, then $\mathbf{u}
  (\mathbf{P} +y\mathbf{Q})^\ell$ gives the joint distribution of
  $(Y_\ell,N_\ell)$. Marginalizing over $Y_\ell$ through the product
  with $\mathbf{v}$ hence achieves the proof of
  Equation~(\ref{eq:mgfGl}). Equation~(\ref{eq:mgfG}) is then obtained
  simply by exploiting the relation $\sum_{k\geqslant 0}
  {\mathbf{M}^k}=(\mathbf{I}-\mathbf{M})^{-1}$ with
  $\mathbf{M}=z(\mathbf{P}+y\mathbf{Q})$.
\end{proof}

\begin{example}
  Considering the same pattern and associated DFA as in
  Example~\ref{ex:1}, one can directly apply Theorem~\ref{thm:dfa} to
  get the expression of $\mathbf{T}$, the transition matrix of
  $Y_0^\ell$ over $\mathcal{Q}=\{0,1,2,3,4,5,6,7,8,9,10,11\}$: $$
  \mathbf{T}=
\left(
    \begin{array}{*{12}c}
     
      \pi_{{\tt B},{\tt B}}&\pi_{{\tt B},{\tt A}}&0&0&0&0&0&0&0&0&0&0\\
     
      0&\pi_{{\tt A},{\tt A}}&\pi_{{\tt A},{\tt B}}&0&0&0&0&0&0&0&0&0\\
     
      0&0&0&\pi_{{\tt B},{\tt A}}&\pi_{{\tt B},{\tt B}}&0&0&0&0&0&0&0\\
     
      0&0&\pi_{{\tt A},{\tt B}}&0&\pi_{{\tt A},{\tt A}}&0&0&0&0&0&0&0\\
     
      \pi_{{\tt B},{\tt B}}&0&0&0&0&\pi_{{\tt B},{\tt A}}&0&0&0&0&0&0\\
     
      0&0&\pi_{{\tt A},{\tt B}}&0&0&0&\pi_{{\tt A},{\tt A}}&0&0&0&0&0\\
     
      0&0&0&0&0&0&0&\pi_{{\tt A},{\tt A}}&\pi_{{\tt A},{\tt B}}&0&0&0\\
     
      0&0&\pi_{{\tt A},{\tt B}}&0&0&0&0&0&0&\pi_{{\tt A},{\tt A}}&0&0\\
     
      0&0&0&0&\pi_{{\tt A},{\tt B}}&0&0&0&0&0&\pi_{{\tt A},{\tt A}}&0\\
     
      0&\pi_{{\tt A},{\tt A}}&0&0&0&0&0&0&0&0&0&\pi_{{\tt A},{\tt B}}\\
     
      0&0&0&0&0&\pi_{{\tt A},{\tt A}}&0&0&0&0&0&\pi_{{\tt A},{\tt B}}\\
     
      0&0&0&\pi_{{\tt B},{\tt A}}&\pi_{{\tt B},{\tt B}}&0&0&0&0&0&0&0\\
    \end{array}
  \right)
$$
where $\pi_{a,b}= \mathbb{P} (X_2=b | X_1 = a)$ for all $a,b \in \{{\tt A},{\tt B}\}$.
\end{example}

\section{Partial recursion}\label{sec:partial}

We want here to focus directly on the expression of $G_\ell(y)$ in
Equation~(\ref{eq:mgfGl}) rather than exploiting the bivariate
expression $G(y,z)$ of Equation~(\ref{eq:mgfG}). A straightforward approach
consists in computing recursively $\mathbf{u} (\mathbf P + y \mathbf
Q)^i$ (or conversely $(\mathbf P + y \mathbf Q)^i \mathbf v$) up to
$i=\ell-m$ taking full advantage of the sparse structure of matrices
$\mathbf P$ and $\mathbf Q$ at each step. This is typically what is
suggested in \citet{Nue06}. The resulting complexity to compute
$\mathbb{P}(N_\ell = n)$ is then $O(\Omega \times n \times \ell)$ in
time and $O(\Omega \times n)$ in space, where $\Omega=R \times
|\mathcal{A}|$ is the number of nonzero elements in
$\mathbf{P}+\mathbf{Q}$. This straightforward (but effective) approach
is easy to implement and is from now on referred as the ``full recursion''
Algorithm.

Another appealing idea is to compute directly the matrix $(\mathbf P
+ y \mathbf Q)^{\ell-m}$ using a classical dyadic decomposition of
$\ell-m$. This is typically the approach suggested by
\cite{ribeca}. The resulting complexity to obtain $\mathbb{P}(N_\ell =
n)$ is $O(R^3 \times n^2 \times \log \ell )$ in time and $O(R^2 \times
n \times \log \ell )$ in space. The algorithm can be further refined
by using FFT-polynomial products (and a very careful implementation)
in order to replace the quadratic complexity in $n$ by $O(n \log
n)$. The resulting algorithm however suffers from numerical instabilities
when considering the tail distribution events and is therefore not
suitable for computing extreme $p$-values. If this approach might offer
interesting performance for relatively small values of $R$ and $n$,
its main drawback is that it totally ignores the sparse structure of
the matrices and therefore fails to deal with highly complex patterns
(large $R$).

Here we want to introduce another approach that fully takes advantage
of the sparsity of the problem to display a linear complexity with $R$
(as with full recursion) but also dramatically reduces the
complexity in terms of the sequence length $\ell$ in order to be suitable for
large scale problems.

From now on, let us assume that $\mathbf{P}$ is an irreducible and
aperiodic \greg{positive} matrix and we denote by $\lambda$ its largest eigenvalue (we
know that $\lambda>0$ thanks to Perron-Frob\'enius). Let us define
$\widetilde{\mathbf{P}}\stackrel{\text{def}}{=}\mathbf{P}/\lambda$ and
$\widetilde{\mathbf{Q}}\stackrel{\text{def}}{=}\mathbf{Q}/\lambda$.

For all $i\geqslant 0$ and all $k \geqslant 0$ we consider the \greg{dimension $R$ column-vector}
$\mathbf{F}_k(i)\stackrel{\text{def}}{=}[y^k](\widetilde{\mathbf{P}}+y\widetilde{\mathbf{Q}})^i\mathbf{v}$.
By convention, $\mathbf{F}_k(i)=\mathbf{0}$ if $i<0$. It is then
possible to derive from Equation~(\ref{eq:mgfGl}) that
$\mathbb{P}(N_\ell=k)=[y^k]F(y)=\lambda^{\ell-m}\mathbf{u}
\mathbf{F}_k(\ell-m)$. Additionally, we recursively define the
\greg{dimension $R$ column-vector} $\mathbf{D}_j^k(i)$ for all $k,i,j \geqslant 0$ by
$\mathbf{D}_k^0(i)\stackrel{\text{def}}{=}\mathbf{F}_k(i)$ and, if $i
\geqslant 1$ and $j \geqslant 1$,
$\mathbf{D}_k^{j}(i)\stackrel{\text{def}}{=}\mathbf{D}_k^{j-1}(i)-\mathbf{D}_k^{j-1}(i-1)$
so that $\mathbf{D}_k^j(i)=\sum_{\delta=0}^{j} (-1)^\delta {j \choose
  \delta} \mathbf{F}_k(i-\delta)$.

\begin{lemma}\label{lemma:1}
  For all $j \geqslant 0$, $k \geqslant 1$, and $i\geqslant j$ we have:
  \begin{equation}
    \mathbf{D}_0^j(i+1)=\widetilde{\mathbf{P}}\mathbf{D}_0^j(i)
    \quad\text{and}\quad
    \mathbf{D}_k^j(i+1)=\widetilde{\mathbf{P}}\mathbf{D}_k^j(i)
    +\widetilde{\mathbf{Q}}\mathbf{D}_{k-1}^j(i).
  \end{equation}
\end{lemma}
\begin{proof}
  The results for $j=0$ come directly from the definition of
  $\mathbf{D}_k^0(i)=\mathbf{F}_k(i)$, the rest of the proof is achieved
  by induction.
\end{proof}

From now on, all asymptotic developments in the current section are
supposed to be valid for $i \rightarrow +\infty$.

\begin{lemma}\label{lemma:2}
  For all $k \geqslant 0$, \greg{there exists a dimension $R$
    column-vector $\mathbf{C}_k$ such that}:
  \begin{equation}\label{eq:Ck}
    \mathbf{D}_k^k(i)=\mathbf{C}_k +O(\nu^{i/(k+1)})
  \end{equation}
where  $\nu$ denotes  the magnitude of the second eigenvalue of
$\widetilde{\mathbf{P}}$ when we order the eigenvalues by decreasing magnitude.
\end{lemma}
\begin{proof}
  It is clear that $\mathbf{D}_0^0(i)=\widetilde{\mathbf{P}}^i
  \mathbf{v}$, elementary linear algebra hence proves the lemma for
  $k=0$ with $\mathbf{C}_0=\widetilde{\mathbf{P}}^\infty
  \mathbf{v}$. We assume now that Equation~(\ref{eq:Ck}) is proved
  up to a fixed $k-1 \geqslant 0$. A recursive application of
  Lemma~\ref{lemma:1} gives for all $\alpha \geqslant k$ and $i
  \geqslant 0$ that
  $\mathbf{D}_k^k(i+\alpha)=\widetilde{\mathbf{P}}^i\mathbf{D}_k^k(\alpha)
  + \sum_{j=1}^{i} \mathbf{P}^{i-j} \widetilde{\mathbf{Q}}
  \mathbf{D}_{k-1}^k(j-1+\alpha)$.\\
  Thanks to \jgdhol{Equation~(\ref{eq:Ck})}
  it is clear that the second term of this sum is a
  $O(\nu^{\alpha/k})$, and we then have
  $\mathbf{D}_k^k(i+\alpha)=\widetilde{\mathbf{P}}^\infty\mathbf{D}_k^k(\alpha)+O(\nu^i)+O(\nu^{\alpha/k})$. Therefore
  we have
  $\mathbf{D}_k^{k+1}(i+\alpha)=O(\nu^i)+O(\nu^{\alpha/k})$. For any
  $j \geqslant k+1$, if we set $\alpha=j\frac{k}{k+1}$, then we obtain
  $\mathbf{D}_k^{k+1}(j)=O(\nu^{j/(k+1)})$, which finishes the proof.
\end{proof}

\begin{proposition}
  For all $j \geqslant 0$, $\alpha \geqslant 0$ and $i \geqslant j+\alpha$ we have:
  \begin{equation}\label{eq:Dk}
  \mathbf{D}_k^j(i)=\sum_{j'=j}^{k} {i-\alpha-k+j'-j  \choose k-j} \mathbf{D}_k^{j'}(\alpha+j')+{i-\alpha-j \choose k-j} O\left(\nu^{\alpha/(k+1)}\right)
  \end{equation}
and in particular for $j=0$ we have:
  \begin{equation}\label{eq:Fk}
  \mathbf{F}_k(i)=\sum_{j'=0}^{k} {i-\alpha-k+j'  \choose k} \mathbf{D}_k^{j'}(\alpha+j')+{i-\alpha \choose k} O\left(\nu^{\alpha/(k+1)}\right).
  \end{equation}
\end{proposition}
\begin{proof}
  This is proved by induction, using repetitively Lemma~\ref{lemma:2}
  and the fact that
  $\mathbf{D}_k^j(i)=\mathbf{D}_k^j(\alpha+j)+\mathbf{D}_k^{j+1}(\alpha+j+1)+\ldots
  + \mathbf{D}_k^{j+1}(i)$.
\end{proof}

\begin{algorithm}[htp]\begin{small}
  \begin{algorithmic}
    \REQUIRE{The matrices $\mathbf{P},\mathbf{Q}$, the vectors
      $\mathbf{u},\mathbf{v}$, $n+1$ column-vectors
      $\mathbf{\Delta}_0$, $\mathbf{\Delta}_1$, $\ldots$, $\mathbf{\Delta}_n$ of dimension $R$, $n+1$ real numbers $R_0,R_1,\ldots,R_n$, and a real number $C$}
    \STATE \COMMENT{Compute $\lambda$ through the power method}
    \STATE $\mathbf{\Delta}_0 \leftarrow \mathbf{v}$, $\lambda \leftarrow 1$
    \WHILE{$\lambda$ has not converged with relative precision $\varepsilon$}
    \STATE $\lambda \leftarrow \mathbf{P}\mathbf{\Delta}_0 / \mathbf{\Delta}_0$ (point-wise division) and $\mathbf{\Delta}_0 \leftarrow \mathbf{P} \mathbf{\Delta}_0$
    \ENDWHILE
    \STATE \COMMENT{Normalize $\mathbf{P}$ and $\mathbf{Q}$}
    \STATE $\mathbf{P} \leftarrow \mathbf{P}/\lambda$ and $\mathbf{Q} \leftarrow \mathbf{Q}/\lambda$
    \STATE \COMMENT{Compute $\alpha$ such as $\mathbf{C}_n=\mathbf{D}_n^n(\alpha)$}
    \STATE $\mathbf{\Delta}_0 \leftarrow \mathbf{v}$
    \FOR{$i=1\ldots n$}
    \FOR{$k=i \ldots 1$}
    \STATE $\mathbf{\Delta}_k \leftarrow \mathbf{P}\mathbf{\Delta}_k +  \mathbf{Q}\mathbf{\Delta}_{k-1}-\mathbf{\Delta}_k$ \COMMENT{so that $\mathbf{\Delta}_k$ now contains $\mathbf{D}_k^i(i)$}
    \ENDFOR
    \STATE $\mathbf{\Delta}_0 \leftarrow \mathbf{P}\mathbf{\Delta}_0-\mathbf{\Delta}_0$ \COMMENT{so that $\mathbf{\Delta}_0$ now contains $\mathbf{D}_0^i(i)$}
    \ENDFOR
    \FOR{$i=n+1\ldots \ell-m$}
    \FOR{$k=n \ldots 1$}
    \STATE $\mathbf{\Delta}_k \leftarrow \mathbf{P}\mathbf{\Delta}_k +  \mathbf{Q}\mathbf{\Delta}_{k-1}$ \COMMENT{so that $\mathbf{\Delta}_k$ now contains $\mathbf{D}_k^n(i)$}
    \ENDFOR
    \STATE $\mathbf{\Delta}_0 \leftarrow \mathbf{P}\mathbf{\Delta}_0$ \COMMENT{so that $\mathbf{\Delta}_0$ now contains $\mathbf{D}_0^n(i)$}
    \IF {$\mathbf{\Delta}_n$ as converged towards $\mathbf{C}_n$ with
      relative precision $\eta$} \STATE $\alpha \leftarrow i$ and {\bf
      break} \ENDIF
    \ENDFOR
    \STATE \COMMENT{Compute all $\mathbf{D}_k^0(\alpha-n)$ for $0 \leqslant k \leqslant n$}
    \STATE $\mathbf{\Delta}_0 \leftarrow \mathbf{v}$
    \FOR{$i=1\ldots \alpha-n$}
    \FOR{$k=n \ldots 1$}
    \STATE $\mathbf{\Delta}_k \leftarrow \mathbf{P}\mathbf{\Delta}_k +  \mathbf{Q}\mathbf{\Delta}_{k-1}$ \COMMENT{so that $\mathbf{\Delta}_k$ now contains $\mathbf{D}_k^0(i)$}
    \ENDFOR
    \STATE $\mathbf{\Delta}_0 \leftarrow \mathbf{P}\mathbf{\Delta}_0$ \COMMENT{so that $\mathbf{\Delta}_0$ now contains $\mathbf{D}_0^0(i)$}
    \ENDFOR
    \STATE \COMMENT{Compute $R_k\stackrel{\text{def}}{=}P_{k,\ell-m}(\alpha-n+k)$ for all $0 \leqslant k \leqslant n$}
    \FOR{$k=0 \ldots n$}
    \STATE $R_k \leftarrow \mathbf{u}\mathbf{\Delta}_k$
    \ENDFOR
    \STATE $C \leftarrow 1.0$
    \FOR{$k=1 \ldots n$}
    \FOR{$j=n \ldots 1$}
    \STATE $\mathbf{\Delta}_j \leftarrow \mathbf{P}\mathbf{\Delta}_j +  \mathbf{Q}\mathbf{\Delta}_{j-1}-\mathbf{P}\mathbf{\Delta}_j$ \COMMENT{so that $\mathbf{\Delta}_j$ now contains $\mathbf{D}_j^k(\alpha-n+k)$}
    \ENDFOR
    \STATE $\mathbf{\Delta}_0 \leftarrow \mathbf{P}\mathbf{\Delta}_0-\mathbf{\Delta}_0$ \COMMENT{so that $\mathbf{\Delta}_0$ now contains $\mathbf{D}_0^k(\alpha-n+k)$}
    \STATE $C \leftarrow C \times (\ell-d-\alpha+n-k+1)/k$
    \FOR{$j=k \ldots n$}
    \STATE $R_k \leftarrow R_k+ C \mathbf{u}\mathbf{\Delta}_k$
    \ENDFOR
    \ENDFOR
    \ENSURE{return $\mathbb{P}(N_\ell=k)=R_k$ for all $0 \leqslant k \leqslant n$}
  \end{algorithmic}
  \end{small}
  \caption{Compute $\mathbb{P}(N_\ell=k)$ with relative error $\eta$
    for all $0 \leqslant k \leqslant n$ using floating point
    arithmetic with a relative error $\varepsilon$
    ($<\eta$).}\label{algo3}
\end{algorithm}

\begin{corollary}
  For any $\alpha \geqslant 0$, for any $k \geqslant 0$ and $\ell -m \geqslant 0$ we have:
  \begin{multline*}
  \mathbb{P}(N_\ell=k)= \underbrace{\lambda^{\ell-m} \sum_{j'=0}^{k} {\ell-m-\alpha-k+j'  \choose k} \mathbf{u}\mathbf{D}_k^{j'}(\alpha+j')}_{P_{k,\ell-m}(\alpha)}\\+\lambda^{\ell-m}{\ell-m-\alpha \choose k} O\left(\nu^{\alpha/(k+1)}\right)
  \end{multline*}
  and $P_{k,\ell-m}(\alpha)$ approximates $\mathbb{P}(N_\ell=k)$ with a relative error of:
  $$
  \left| \frac{P_{k,\ell-m}(\alpha)-\mathbb{P}(N_\ell=k)}{\mathbb{P}(N_\ell=k)} \right | =\frac{\lambda^{\ell-m}}{\mathbb{P}(N_\ell=k)}{\ell-m-\alpha \choose k} O\left(\nu^{\alpha/(k+1)}\right).
  $$
\end{corollary}
These results lead to Algorithm~\ref{algo3} \greg{which allows to compute $\mathbb{P}(N_\ell=k)$ for all $0 \leqslant k \leqslant n$ for a given $n \geqslant 0$.}
 The workspace complexity of this algorithm is $O(n \times R)$ and
since all matrix vector products exploit the sparse structure of the
matrices, the time complexity is $O(\alpha \times n \times
|\mathcal{A}| \times R)$ with 
\greg{$\alpha=O\left(n^2\log(\ell)/\log(\nu^{-1})\right)$

where $\nu$ is the magnitude of the second
largest eigenvalue.}

\section{High-order lifting}\label{sec:highorder}

An alternative appealing idea consists of using Equation (\ref{eq:mgfG}) to
compute explicitly the rational function $G(y,z)$ and then \jgdhol{perform}
(fast) Taylor expansions, first in $z$, then in $y$ in order to get
the appropriate 
values of $\mathbb{P}(N_\ell=n)$ \citep{nicodeme,lladser}.

\jgdhol{To compute $G(y,z)$}, a really naive approach would solve the
bivariate polynomial system \jgdhol{(\ref{eq:mgfG})}
over the rationals. In the solution, the polynomials would be of
degree at most $R$ in each variable and the rationals of size at most
$R\log_2(R)$, thus yielding a binary cost of $O(R^9\log^2_2(R))$ already for the
computation of $G(y,z)$ and a further cost of the same magnitude for
the extraction of the coefficient of degree $\ell$ in the
development \jgdhol{with, e.g., linear system solving}.

Since \jgdhol{the complexity of computing this generating function} is
usually prohibitive, we first  use modular methods and Chinese
remaindering to compute with polynomials and rationals so that the
cost of the arithmetic remains linear. 
\jgdhol{Also, to take advantage}
of the sparsity we do not invert the matrix but rather compute
directly a rational reconstruction of the solution from the first
iterates of the series, Equation (\ref{eq:mgfGl}). We thus only use sparse matrix-vector products
and do not \jgdhol{fill in} the matrix. Note that \jgdhol{this direct rational
reconstruction} is equivalent to solving the system using iterative
$z$-adic methods.
Finally, we compute only small
chunks of the  
development of Equation (\ref{eq:mgfG}) using the ``high-order'' lifting of 
\cite{Storjohann:2003:jsc}, or the method of \cite{Fiduccia:1985:EFLR}, {\em modulo} $y^{\nu+1}$.
We aim to compute only all the coefficients
$g_{\ell,n} = \mathbb P(N_\ell=n)$ for a given $\ell$ and a given $n\in
[\mu,\nu]$, for an interval $[\mu,\nu]$ with a small $\nu$, but where $\ell$ is
potentially large. Thus we want to avoid computing the whole Taylor
development of the fraction.

Indeed, let $G(y,z)= \frac{B(y,z)}{A(y,z)}$ with $B, A \in \Q[y,z]$ of
degrees $d_{Az}=\deg(A,z)$ and $d_{Bz}=\deg(B,z) \leq
d_{Az}-1$. Overall let us denote by $d$ a bound on $d_{Az}$ and thus
on $d_{Bz}$.  We can assume that $A$ and $B$ are polynomials since we
can always pre-multiply them both by the lowest common multiple of
their denominators.

\jgdhol{Thus we write $B = \sum_{i=0}^{d_{Bz}} b_i(y) z^i$ and, if we denote by
$[z^i]P(z)$ the $i$-th coefficient of the polynomial $P$, then we have:}
$$[z^\ell]G(y,z) = \sum_{i=0}^{d_{Bz}} b_i(y) \times
[z^{\ell-i}]\left(\frac{1}{A(y,z)}\right)$$

Now, the idea is that for a given coefficient $\ell$, the only
coefficients of the development of $1/A$ that are needed are the
coefficients of order $\ell-m$ to $\ell$,
denoted by $[A^{-1}]_{\ell-m}^\ell$.
This is precisely what Storjohann's ``High-order'' lifting can compute
quickly.\\

We will use several Chinese remaindering and rational
reconstructions.

In our cases, we have a series $\sum^\infty g_i z^i$ which is actually
equal to a fraction $\frac{B(z)}{A(z)}$ by definition. Therefore,
provided that we consider sufficiently many of the first coefficients in 
the development, the rational reconstruction of the truncated series
$\sum^{2d} g_i z^i$,
even with rational polynomial in $y$ as coefficients, {\em will} eventually yield $A(z)$ and $B(z)$ as
solutions.

\subsection{Rational reconstruction}

The first step is thus to recover both polynomials $B$ and $A$ from
the series development of Equation (\ref{eq:mgfG}).
Now, one could compute the whole rational
reconstruction of a polynomial over the domain of the rationals, and
then use $d^2\times d^2$ operations for the domain arithmetic, which
would yield a $d^6$ complexity to compute all the $d^2$ coefficients. We
rather use two modular projections, one for the rational
coefficients, one for the polynomials in $y$, in order to reduce the
cost of the arithmetic to only $d^2$. Then the overall cost will
be dominated by the cost of the computation of only $d$ coefficients
of the series, as shown in \jgdhol{Proposition~\ref{prop:fr}}.

Our algorithm has then four main steps:
\begin{enumerate}[1)]
\item We take the series in $z$ modulo some prime numbers (below $2^\gamma$ where $\gamma$ \jgdhol{is, e.g., close to word size}) and some evaluation
  points on the variable $y$;
\item We perform a univariate polynomial fraction reconstruction in $z$ over
  each prime field and for each evaluation point;
\item We interpolate the polynomials in $y$ over the prime field from
  their evaluations;
\item We finally Chinese remainder and rational reconstruct the
  rational coefficients from their modular values.
\end{enumerate}
\newcommand{\enn}{n}

The details of the approach are given in Algorithm~\ref{alg:rr} whose complexity is given by the following proposition.

\begin{algorithm}[htp]
\caption{Modular Bivariate fraction reconstruction over the rationals.}\label{alg:rr}
\begin{algorithmic}[1]
\REQUIRE The matrices $\mathbf{P}$, $\mathbf{Q}$ and the row and
column vectors $\mathbf{u}$, $\mathbf{v}$ defining $G(y,z)$.
\ENSURE Polynomials $B(y,z)$ and $A(y,z)$ of degree 
$\leq d$ with $G(y,z)=\frac{B(y,z)}{A(y,z)}$~\,
\STATE Let $G(y,z)=\mathbf{u} \mathbf{v}$, $v_0(y)=\mathbf{v}$;
\FOR{$\enn=1$ \UPTO $2d$}
	\STATE\label{lin:apply} $v_\enn(y)=(\mathbf{P}+y\mathbf{Q})v_{\enn-1}(y)$;
        \STATE\label{lin:dotp} $G(y,z)=G(y,z)+\mathbf{u} v_\enn(y) z^\enn$;
\ENDFOR
\STATE Let $\overline{d} = (2d+2) \log_{2^\gamma}( ||\mathbf{P},\mathbf{Q},\mathbf{u},\mathbf{v}||_\infty )$;
\FOR{$i=0$ \UPTO $\overline{d}$}
	\STATE Let $p_i\geq 2^\gamma>d$ be a prime, coprime with
        $p_0,\ldots,p_{i-1}$;
        \STATE $G_i(y,z) = G(y,z) \mod p_i$;
        \FOR{$j=0$ \UPTO $d$}
        	\STATE Let $y_j$ be an element modulo $p_i$, distinct from $y_0,\ldots,y_{j-1}$;
                \STATE $G_{i,j}(z) = G_i(y_j,z) \mod p_i$;
                \STATE\label{lin:fr}
                $\frac{B_{i,j}(z)}{A_{i,j}(z)}=${\bf FractionReconstruction}$(G_{i,j}(z)) \mod p_i \mod (y-y_j)$;
\\\hfill\COMMENT{$B_{i,j}(z)=\sum_{\enn=0}^d B_{i,j,\enn}z^\enn$ and
                $A_{i,j}(z)=\sum_{\enn=0}^d A_{i,j,\enn}z^\enn$~\,}

        \ENDFOR
        \FOR{$\enn=0$ \UPTO $d$}
        	\STATE\label{lin:intb} {\bf Interpolate} $B_{i,\enn}(y) \mod p_i$ from $B_{i,j,\enn}$ for $j=0..d$;
        	\STATE\label{lin:inta} {\bf Interpolate} $A_{i,\enn}(y) \mod p_i$ from $A_{i,j,\enn}$ for $j=0..d$;
        \ENDFOR                
                \hfill\COMMENT{$B_{i}(y,z)=\sum_{\enn=0}^d B_{i,\enn}(y)z^\enn$ and
                $A_{i}(y,z)=\sum_{\enn=0}^d A_{i,\enn}(y)z^\enn$~\,}
\ENDFOR
\FOR{$\enn=0$ \UPTO $d$}
	\FOR{$l=0$ \UPTO $d$}
		\STATE\label{lin:crab} {\bf ChineseRemainder} 
                $[y^\enn z^l]B(y,z)$ from
                $[y^\enn z^l]B_i(y,z)$ for $i=0..\overline{d}$;
		\STATE\label{lin:craa} {\bf ChineseRemainder} 
                $[y^\enn z^l]A(y,z)$ from
                $[y^\enn z^l]A_i(y,z)$ for $i=0..\overline{d}$;
                \STATE\label{lin:rr} {\bf RationalReconstruct} both obtained coefficients;
        \ENDFOR                
\ENDFOR
\end{algorithmic}
\end{algorithm}

\begin{proposition}\label{prop:fr}
Let $d=\max\{d_A,d_B\}$ be the degree in $z$ and $\nu$ be the largest
degree in $y$ of the coefficients of $A$ and $B$. Let $\Omega=|{\mathcal
  A}|R$ be the total number of nonzero coefficients of the $R\times R$
matrices $\mathbf{P}$ and $\mathbf{Q}$. If the coefficients of the matrices $\mathbf{P}$, $\mathbf{Q}$,
$\mathbf{u}$ and $\mathbf{v}$, and $|{\mathcal A}|$ are constants, then
the cost of the computation of $B$ and $A$ in 
Algorithm~\ref{alg:rr} is  
$$O\left( d^3 R \log(R) \right),$$
where the intermediate memory requirements are of the same order of magnitude.
\end{proposition}
\begin{proof}
Polynomial fraction reconstruction of degree $d$ requires $2d$
coefficients. The computation of one coefficient of the evaluated
series modulo costs one matrix-vector product, $\Omega$ word operations,
and a dot product of size $R\leq \Omega$. 
By definition $\deg(g_j(y))=\deg((\mathbf{P}+y\mathbf{Q})^j)\leq j$, thus
$\nu\leq 2d$ and, similarly, the size of the rational coefficients
is bounded by
$(2d+2)\log(R ||\mathbf{P},\mathbf{Q},\mathbf{u},\mathbf{v}||_\infty^2)$.

Thus steps 3 and 4 in Algorithm
\ref{alg:rr} cost $$\sum_{\enn=0}^d O(\Omega \enn^2\log(R
||\mathbf{P},\mathbf{Q},\mathbf{u},\mathbf{v}||_\infty^2)=O(d^3
\Omega\log(R ||\mathbf{P},\mathbf{Q},\mathbf{u},\mathbf{v}||_\infty^2)$$ operations.

Then a fraction reconstruction of degree $d$ costs less than
$O(d^2)$ operations by Berlekamp-Massey or the extended Euclidean
algorithm and an interpolation of $d$ points costs less than $O(d^2)$
operations so that, overall, 

steps 12, 14 and 15
cost less than $O(d^4)$ operations.
Then Chinese remaindering and rational reconstruction of
size $d$ costs less than $O(d^2)$
for an overall cost of $$O(d^4\log^2(||\mathbf{P},\mathbf{Q},\mathbf{u},\mathbf{v}||_\infty).$$
As $d\leq R\leq \Omega$, if $\log_{2^\gamma}^2(||\mathbf{P},\mathbf{Q},\mathbf{u},\mathbf{v}||_\infty) =
O(1)$, then this latter term is dominated by $O( d^3\Omega\log(R)
)$. Finally, if $|{\mathcal A}|$ is constant, we have that $\Omega = O(R)$.

Now, the memory requirements are bounded by those of the series.
The vector  $v_\enn(y)$ is of size $R$, of degree $\enn$ in $y$ with
rational coefficients of size $\enn \log(R
||\mathbf{P},\mathbf{Q},\mathbf{u},\mathbf{v}||_\infty^2)$. Thus
$v_\enn(y)$ and the dot product $\mathbf{u} v_\enn(y)$ are of size $O( R \enn^2 
\log(R) )$ so that $G(y,z)= \sum_{\enn=0}^{2d} \mathbf{u}
v_\enn(y)$ is $O(R\log(R) d^3)$.
\end{proof}

Thus the dominant computation in Algorithm~\ref{alg:rr} is the computation
of the first terms of the series $G(y,z)$.

\subsection{Early termination strategy for the determination of the
  fraction degrees} 
\jgdhol{There remains to determine the value of the degree $d=\max\{d_A,d_B\}$
for the actual solutions $A$ and $B$}. As the series is
the solution of a polynomial linear system of size $R$ and degree $1$,
the determinant, and thus the denominator and numerator of the
solution are of degree bounded by $R$.
Now in practice, we will see that very often this degree is much
smaller than $R$. As the complexity is cubic in $d$ it is of major
importance to determine as accurately as possible this $d$ beforehand.

The strategy we propose is an early termination, probabilistic of
Las Vegas type, i.e. it is 
always correct, but sometimes slow.
We reconstruct the rational fraction only from a small number
$d_0$ of iterations, and then again after $d_0+1$ iterations. If both
fractions are identical with numerator and denominator of degrees
strictly less than $d_0$, then there is a good chance that the
recovered fraction is the actual one. This can be checked by just
applying $A$ to the obtained guess and verifying that it
corresponds to the right-hand side.
If the fractions disagree or if the check fails, one can then try
again after some more iterations. 

In our bivariate case over the rationals, we have a very fast strategy
which consists in finding first the degree at a single evaluation point
modulo a single prime and e.g. roughly doubling the number of
iterations at each failure. This search thus requires less than
$2\times 2d$ iterations and has then a marginal cost of $O(d\Omega+d^2)$.

\subsection{High-order lifting for polynomial fraction development} 
Once the bivariate fraction $\frac{B}{A}$ is recovered, the next step
is to compute the coefficients of degree $\ell \in [\alpha,\beta]$ of its
series development. 
The idea of the high-order lifting of \cite{Storjohann:2003:jsc} 
is to make use of some particular points in the
development, that are computable independently.
Then these points will enable fast computation of only
high-order terms of the development and not of the whole series.
In the following, we call these points {\em  residues}.

We first need the fundamental Lemma~\ref{lem:res}.
\begin{lemma}\label{lem:res} Let $A,B \in \mathbb{R}[z]$ be of respective degrees
  $d_A$ and $d_B
  \leq d_A-1$. Then for all $\ell \in \N$, there exists $B_\ell \in
  \mathbb{R}[z]$ of degree $d_{B_\ell}\leq d_A-1$ and $(g_i)_{i=0..\ell-1}\in
  \mathbb{R}^\ell$ such that
$$\frac{B(z)}{A(z)} = \sum_{i=0}^{\ell-1} g_i z^i + \frac{B_\ell(z)}{A(z)} z^\ell$$
\end{lemma}
\begin{proof} We use the same construction as~\cite{Salvy:2009:cours7}:
the initial series \jgdhol{rewrites as 
$\frac{B}{A} = \sum_{i=0}^{\infty} g_i z^i =
\sum_{i=0}^{\ell-1} g_i z^i +  z^\ell \sum_{i=0}^{\infty} g_{i+\ell}
z^i$}.
Then let $\overline{B_\ell} = B - A \left(\sum_{i=0}^{\ell-1} g_i
  z^i\right)$.
By construction, $degree(\overline{B_\ell}) = d_A+\ell-1$, but we also
have that $\overline{B_\ell} = z^\ell \sum_{i=0}^{\infty} g_{i+\ell}
z^i$.
We thus let $B_\ell = \overline{B_\ell} z^{-\ell}$ which satisfies the
hypothesis.
\end{proof}

The question is how to compute efficiently the $\ell$th residue
$B_\ell$ defined in Lemma~\ref{lem:res}.
The idea is to use the high-order lifting of
\cite{Storjohann:2003:jsc}.

We follow the presentation of the lifting of \cite{Salvy:2009:cours7} but
define a slightly different bracket notation for chunks of a polynomial:
\begin{equation}\label{eq:chunk}
[A(z)]_\alpha^\beta = a_\alpha z^\alpha + \ldots + a_\beta z^\beta
\end{equation}

Roughly there are two main parts. The first one generalizes the
construction of Lemma~\ref{lem:res} using only $2d$ coefficients of
$A^{-1}$. The second part builds small chunks of size $2d$ of $A^{-1}$
at high-orders, each being close to a power of $2$.

The efficient computation of residues given in Algorithm~\ref{alg:bk}
\jgdhol{takes simple advantage} of the fact that a given residue has a small
degree and depends only
on a small part of the development of $A^{-1}$. We first give a
version where the adequate part of $A$ is given as input. We will
later detail the way to efficiently compute these coefficients.

\begin{algorithm}[ht]
\caption{Residue($A$,$B$,$j$,$V$).}\label{alg:bk}
\begin{algorithmic}[1]
\REQUIRE $A$, $B$, $j$ and $V= [A^{-1}]_{j-2d+1}^{j-1}$.
\ENSURE $B_j$ defined by Lemma~\ref{lem:res}.
\STATE Compute $U \stackrel{\text{def}}{=} [VB]_{j-d}^{j-1}$;
\STATE Return $B_j \stackrel{\text{def}}{=} z^{-j}  [B - AU]_{j}^{j+d-1} $;
\end{algorithmic}
\end{algorithm}

\begin{lemma} The arithmetic complexity of Algorithm~\ref{alg:bk} is
  $2d^2$ operations.
\end{lemma}

Then, we define $\Gamma_\ell$ to be the high-order residue of the
expansion of $A^{-1}$, using Lemma~\ref{lem:res} with $B=1$:
\begin{equation}\label{eq:gamma}
\frac{1}{A(z)} = \sum_{i=0}^{\ell-1} g_i z^i +
\frac{\Gamma_\ell(z)}{A(z)} z^\ell
\end{equation}
The idea of the fast lifting is that when substituting $A^{-1}$ in the
right hand side of \jgdhol{Equation~(\ref{eq:gamma})} by this actual right hand
side, one gets:
$$\sum_{i=0}^{\ell-1} g_i z^i +
\Gamma_\ell(z) \left( \sum_{i=0}^{\ell-1} g_i z^i +
\frac{\Gamma_\ell(z)}{A(z)} z^\ell \right) z^\ell
= \sum_{i=0}^{2\ell-1} g_i z^i +
\frac{\Gamma_{2\ell}(z)}{A(z)} z^{2\ell}
$$
This shows that $\Gamma_{2\ell}$ depends only on $\Gamma_\ell$ and of
chunks of $A^{-1}$, of size $d$, at $0$ and around $\Gamma_\ell$; more
generally one gets the following formula:
\begin{equation}\label{eq:invchunk}
\left[A^{-1}\right]_\alpha^\beta = z^\ell \left[ \Gamma_\ell
  [A^{-1}]_{\alpha-\ell-d}^{\beta-\ell} \right]_{\alpha-\ell}^{\beta-\ell}
\end{equation}
\jgdhol{This formula states, from Equation~(\ref{eq:gamma})}, that the Taylor
coefficients of order greater than
$\ell$, can also be recovered from the Taylor development of
$\frac{\Gamma_\ell(z)}{A(z)}$.

Then the second part of the high-order lifting is thus
Algorithm~\ref{alg:dol} which gets a small chunk of $A^{-1}$ at a double
order of what it is given as input, as shown in Figure~\ref{fig:dbord}.

\begin{algorithm}[ht]
\caption{Double-Order($S, T, \Gamma,e$).}\label{alg:dol}
\begin{algorithmic}[1]
\REQUIRE $S = [A^{-1}]_{0}^{d-1}$, $V_e = [A^{-1}]_{2^e-2d+1}^{2^e-1}$
and $\Gamma_{2^e-d}$ defined by eq.~(\ref{eq:gamma}).
\ENSURE $\Gamma_{2^{e+1}-d}$ and $V_{e+1} = [A^{-1}]_{2^{e+1}-2d+1}^{2^{e+1}-1}$.
\STATE Compute $V_L \stackrel{\text{def}}{=} z^{2^{e}-d}[ 
\Gamma_{2^{e}-d} V_e ]_{2^{e}-d}^{2^{e}-1}$;
\hfill\COMMENT{eq. (\ref{eq:invchunk})~\,}
\STATE Compute $\Gamma_{2^{e+1}-d}\stackrel{\text{def}}{=} z^{-2^{e+1}+d}[1-A V_L ]_{2^{e+1}-d}^{2^{e+1}-1} $;
\hfill\COMMENT{Residue$(A,1,2^{e+1}-d)$~\,}
\STATE Compute $V_H\stackrel{\text{def}}{=}
z^{2^{e+1}-d} [\Gamma_{2^{e+1}-d}S] _{0}^{d-1} $;
\hfill\COMMENT{eq. (\ref{eq:invchunk})~\,}
\STATE Return $\Gamma_{2^{e+1}-d}$ and
$V_{e+1}=[A^{-1}]_{2^{e+1}-2d+1}^{2^{e+1}-1} \stackrel{\text{def}}{=}
[V_L]_{2^{e+1}-2d+1}^{2^{e+1}-d-1} +V_H$.
\end{algorithmic}
\end{algorithm}

\begin{figure}[htbp]
  \begin{center}
    \includegraphics[width=\textwidth]{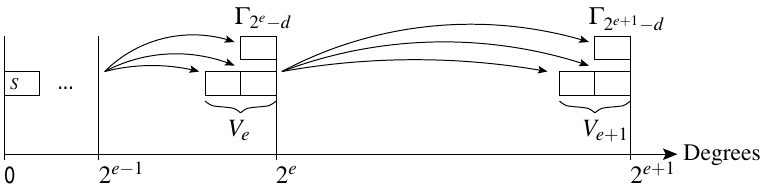}
\caption{Computing chunks at double order.}\label{fig:dbord}
  \end{center}
\end{figure}

\begin{lemma} The arithmetic complexity of Algorithm~\ref{alg:dol} is
  $3d^2$ operations.
\end{lemma}
\begin{proof}
\jgdhol{Below are the complexity of the successive steps of Algorithm~\ref{alg:dol}.}
\begin{enumerate}[1)]
\item One truncated polynomial multiplication of degree $d-1$, complexity $d^2$;
\item One truncated polynomial multiplication of degree $d-1$, complexity $d^2$;
\item One truncated polynomial multiplication of degree $d-1$, complexity $d^2$.
\item No-op, $V^{(L)}$ and $V^{(H)}$ do not overlap.
\end{enumerate}
\end{proof}

Then Algorithm~\ref{alg:hol} gives the complete precomputation to get
a sequence of doubling order $\Gamma$'s which will enable fast
computations of the chunks of the Taylor expansion.

\begin{algorithm}[ht]
\caption{High-Order($A,\alpha,\beta$).}\label{alg:hol}
\begin{algorithmic}[1]
\REQUIRE A polynomial $A(z)$ of degree $d$.
\REQUIRE A valuation $\alpha$ and a degree $\beta \geq d$.
\ENSURE $e_0$ s.t. $2^{e_0-1}<2d\leq2^{e_0}$ ; $e_\beta$ s.t. 
$2^{e_\beta}\leq \beta+d< 2^{e_\beta+1}$.
\ENSURE The Taylor development of $\frac{1}{A}$ up to
$\delta=\max\{2^{e_0}-1;\beta-\alpha\}$.
\ENSURE $\left(\Gamma_{2^{e_0}-d},\ldots,\Gamma_{2^{e_\beta}-d}\right)$.
\STATE Compute $e_0$ s.t. $2^{e_0-1}<2d\leq2^{e_0}$ ; $e_\beta$ s.t. 
$2^{e_\beta}\leq \beta+d< 2^{e_\beta+1}$;
\STATE Let $\xi_0 = k_{e_0}\stackrel{\text{def}}{=}2^{e_0}-d$ and $\delta\stackrel{\text{def}}{=}\max\{2^{e_0}-1;\beta-\alpha\}$;
\STATE Compute $S\stackrel{\text{def}}{=}[A^{-1}]_0^{\delta}$, via Taylor expansion of $A$;
\STATE $U_0 \stackrel{\text{def}}{=} [A^{-1}]_{\xi_0-d}^{\xi_0-1} = [S]_{\xi_0-d}^{\xi_0-1}$;
\STATE Compute $\Gamma_{\xi_0} \stackrel{\text{def}}{=} z^{-\xi_0}[I - A U_0]_{\xi_0}^{\xi_0+d-1}$;
\hfill\COMMENT{Residue$(A,1,\xi_0)$~\,}
\STATE $V_{e_0}\stackrel{\text{def}}{=}[A^{-1}]_{\xi_0-d+1}^{\xi_0+d-1}=[S]_{2^{e_0}-2d+1}^{2^{e_0}-1}$;
\FOR {$i=e_0+1$ \UPTO $e_\beta$}
	\STATE $k_i\stackrel{\text{def}}{=}2^i-d$;
	\STATE $\left(\Gamma_{k_i};V_{i}\right) \stackrel{\text{def}}{=}
        $Double-Order$([A^{-1}]_0^{d-1},V_{i-1},
        \Gamma_{k_{i-1}},i-1)$;
\ENDFOR
\end{algorithmic}
\end{algorithm}

Figure~\ref{tab:ho50} shows which high-order terms are recovered during
these giant steps of precomputation, for a computation using $50$
precomputed terms of the Taylor development with
Algorithm~\ref{alg:hol} and $degree(A)=6$.

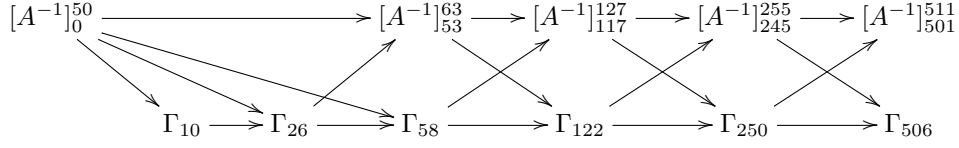
\begin{figure}[htp]\center
$$ \xymatrix@C=20pt@R=25pt{
[A^{-1}]_0^{50} \ar[rd]\ar[rrd]\ar[rrrd]\ar[rrr] & & & [A^{-1}]_{53}^{63}\ar[r]\ar[rd] & [A^{-1}]_{117}^{127}\ar[r]\ar[rd] & [A^{-1}]_{245}^{255}\ar[r]\ar[rd] & [A^{-1}]_{501}^{511} \\
 & \Gamma_{10} \ar[r] & \Gamma_{26}\ar[r]\ar[ru]& \Gamma_{58}\ar[r]\ar[ru] & \Gamma_{122}\ar[r]\ar[ru] & \Gamma_{250}\ar[r]\ar[ru] & \Gamma_{506} \\
} 
$$
\caption{High-Order lifting to $2^{\lfloor \log_2(\beta) \rfloor}-1$
  computing $\Gamma_{2^4-6}$, $\Gamma_{2^5-6}$, $\Gamma_{2^6-6}$,
  $\Gamma_{2^7-6}$, $\Gamma_{2^8-6}$ and $\Gamma_{2^9-6}$.}\label{tab:ho50}
\end{figure}

\begin{lemma}\label{lem:hol}
The arithmetic complexity of Algorithm~\ref{alg:hol} is
  less than: $$3\log_2\left(\frac{\beta+d}{2d}\right)d^2+\max\{4d^2;d(\beta-\alpha+2d)\}$$
\end{lemma}
\begin{proof}
\jgdhol{Below are the complexity of the successive steps of Algorithm~\ref{alg:hol}.}
\begin{itemize}
\item[3.] One Taylor expansion of an inverse of degree $d$ up to the
  degree $\delta\leq \max\{2d;\beta-\alpha\}$, complexity
  $\sum_{i=1}^{d} 2i-1+\sum_{i=d+1}^{\delta} 2d-1\leq d(2\delta-d) \leq \max\{3d^2;d(\beta-\alpha+d)\}$.
\item[5.] One truncated polynomial multiplication of degree $d-1$, complexity $d^2$.
\item[9.] $e_\beta-e_0 \leq \log_2\left(\frac{\beta+d}{2d}\right)$ calls to
  Algorithm~\ref{alg:dol}, complexity bounded by
  $3\log_2\left(\frac{\beta+d}{2d}\right)d^2$.
\end{itemize}
\end{proof}

Once the high-order terms are computed, one can get the development
chunks of $[\frac{B}{A}]_\alpha^\beta$ as shown in
Algorithm~\ref{alg:dchunk}.

\begin{algorithm}[ht]
\caption{DevelChunk$(A,B,(\Gamma_i),S,\alpha,\beta)$.}\label{alg:dchunk}
\begin{algorithmic}[1]
\REQUIRE \jgd{$A$, $B$, $(\Gamma_i)$ as defined in 
\jgdhol{Equation (\ref{eq:gamma}).}}
\REQUIRE \jgdhol{A valuation} $\alpha$, a degree $\beta$ and $S=[A^{-1}]_0^\delta$, with
$\delta \geq \beta-\alpha$.
\ENSURE $[BA^{-1}]_\alpha^\beta$
\IF{$\beta \leq \delta$}
\STATE Return $[BS]_\alpha^\beta$;
\ELSE
\STATE $B_\alpha\stackrel{\text{def}}{=}$Residue$(A,B,(\Gamma_i),S,\alpha)$;
\STATE Return $z^\alpha[B_\alpha S]_0^{\beta-\alpha}$;
\hfill\COMMENT{eq. (\ref{eq:invchunk})~\,}
\ENDIF
\end{algorithmic}
\end{algorithm}

Algorithm~\ref{alg:dchunk} uses a variant of Algorithm~\ref{alg:bk}, which
needs to compute on the fly the chunks of the inverse it
requires. This variant is shown in Algorithm~\ref{alg:residue}.

\begin{algorithm}[ht]
\caption{Residue$(A,B,(\Gamma_i),S,\alpha)$.}\label{alg:residue}
\begin{algorithmic}[1]
\REQUIRE \jgd{$A$, $B$, $(\Gamma_i)$, $S$, $\alpha$ as in \jgdhol{Algorithm~\ref{alg:dchunk}}.}
\ENSURE \jgd{$B_\alpha$ defined by Lemma~\ref{lem:res}.}
\IF{$\alpha = 0$}
\STATE Return $B$;
\ELSE
\STATE $V\stackrel{\text{def}}{=}$InverseChunk$(A,(\Gamma_i),S,\alpha-2d+1,\alpha-1)$;
\STATE $U\stackrel{\text{def}}{=}[V B]_{\alpha-d}^{\alpha-1}$;
\STATE Return $z^{-\alpha}[B-AU]_\alpha^{\alpha+d-1}$.
\hfill\COMMENT {Residue$(A,B,\alpha,V)$~\,}
\ENDIF
\end{algorithmic}
\end{algorithm}

The point is that these chunks of the development of the inverse are
recovered just like the chunks of any fraction, but with some
high-order residues already computed. 
Algorithm~\ref{alg:ichunk} is thus a variant of
Algorithm~\ref{alg:dchunk} with $B=1$ and a special residue call.

\begin{algorithm}[ht]
\caption{InverseChunk$(A,(\Gamma_i),S,\alpha,\beta)$.}\label{alg:ichunk}
\begin{algorithmic}[1]
\REQUIRE \jgd{$A$, $(\Gamma_i)$, $\alpha$, $\beta$,
  $S=[A^{-1}]_0^\delta$ with $\delta\geq\beta-\alpha$ as in \jgdhol{Algorithm~\ref{alg:dchunk}.}}
\ENSURE $[A^{-1}]_\alpha^\beta$
\IF{$\beta \leq \delta$}
\STATE Return $[S]_\alpha^\beta$;
\ELSE
\STATE $\Gamma_\alpha\stackrel{\text{def}}{=}$Get$\Gamma(A,(\Gamma_i),S,\alpha)$;
\hfill\COMMENT{Residue$(A,1,\alpha)$;~\,}
\STATE Return $z^{\alpha} [\Gamma_\alpha S ]_0^{\beta-\alpha}$;
\hfill\COMMENT{eq. (\ref{eq:invchunk})~\,}
\ENDIF
\end{algorithmic}
\end{algorithm}

Algorithm~\ref{alg:getg} is this special residue call, a variant of
Algorithm~\ref{alg:residue}, which uses the precomputed high-order
$\Gamma_i$ to get the actual $\Gamma_\alpha$ it needs, in a logarithmic
binary-like decomposition of $\alpha$.

\begin{algorithm}[ht]
\caption{Get$\Gamma(A,(\Gamma_i),S,\alpha)$.}\label{alg:getg}
\begin{algorithmic}[1]
\REQUIRE \jgd{$A$, $(\Gamma_i)$, $S$, $\alpha$ as in \jgdhol{Algorithm~\ref{alg:dchunk}.}}
\ENSURE $\Gamma_\alpha$;
\IF{$\alpha=0$}
\STATE Return $1$.
\ELSIF{$\alpha\leq \delta$}
\STATE $U\stackrel{\text{def}}{=}[S]_0^{\alpha-1}$;
\STATE Return $z^{-\alpha}[1-AU]_\alpha^{\alpha+d-1}$.
\hfill\COMMENT {Residue$(A,1,\alpha,S)$~\,}
\ELSE
\STATE \jgd{$a = \lfloor\log_2(\alpha+d)\rfloor$;} \COMMENT{so that
  $2^a \leq \alpha+d < 2^{a+1}$}
\STATE \jgd{Return Residue$(A,\Gamma_a,(\Gamma_i),\alpha+d-2^a)$;}
\ENDIF
\end{algorithmic}
\end{algorithm}

We have shown in Figure~\ref{tab:ho50} the high-order
precomputation of the different $\Gamma$'s required 
\jgdhol{for the computation, e.g., of $[BA^{-1}]_{950}^{1000}=[B_{950}
  A^{-1}]_0^{50}$}, with $A$ of degree $6$.
Then, Figure~\ref{tab:chunks1000} shows the
successive baby step calls of residue and inverse chunks needed to
compute the $950$th residue $B_{950}$
\begin{figure}[htp]
\jgd{$$
 \xymatrix@C=20pt@R=30pt{
[BA^{-1}]_{950}^{1000} & *+[F]{[A^{-1}]_{0}^{50}}\ar[l]&
[A^{-1}]_{939}^{949}\ar[ld]|*[@][*.6]!/^-7pt/{\hspace{-28pt}\text{i-chunk}}&
[A^{-1}]_{422}^{432}\ar[ld]|*[@][*.6]!/^-7pt/{\hspace{-28pt}\text{i-chunk}}&
[A^{-1}]_{161}^{171}\ar[ld]|*[@][*.6]!/^-7pt/{\hspace{-28pt}\text{i-chunk}}&
*+[F]{[A^{-1}]_{28}^{38}}\ar[ld]|*[@][*.6]!/^-7pt/{\hspace{-28pt}\text{i-chunk}}\\
 &
 B_{950}\ar[lu]|*[@][*.6]!/_4pt/{\hspace{-34pt}\text{Residue}} & 
 \Gamma_{939}\ar[u]_{\text{Get}\Gamma}&
 \Gamma_{422}\ar[u]_{\text{Get}\Gamma}&
 \Gamma_{161}\ar[u]_{\text{Get}\Gamma}\\
 & & *+[F]{ \Gamma_{506}}\ar[u]_{\text{Residue}}&
 *+[F]{\Gamma_{250}}\ar[u]_{\text{Residue}}&
 *+[F]{\Gamma_{122}}\ar[u]_{\text{Residue}}\\
}$$} 
\caption{\jgd{$DevelChunk$ expansion (alg.~\ref{alg:dchunk} from right to
  left) where some chunks (boxed) were precomputed by high-order
  lifting (alg.~\ref{alg:hol} and
  fig.~\ref{tab:ho50}).}}\label{tab:chunks1000}
\end{figure}
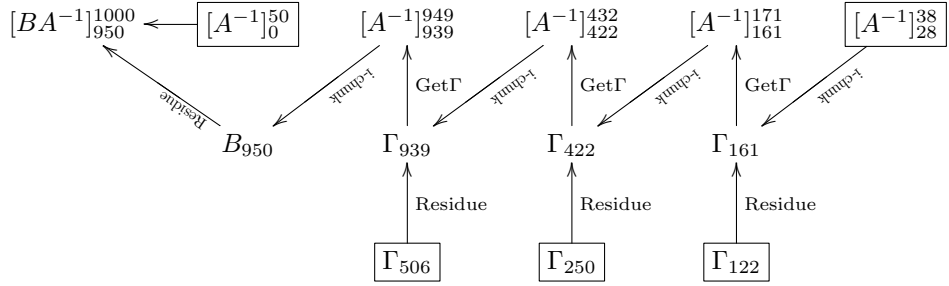

\begin{lemma}\label{lem:chunk} The arithmetic complexity of computing the
  $[\alpha,\beta]$ chunk of the development of $[B
  A^{-1}(z)]_\alpha^\beta$ via Algorithm~\ref{alg:dchunk} is
  less than: $$\log_2\left(\frac{\beta+d}{2d}\right)\left(3d^2+2(\beta-\alpha)^2\right)$$
\end{lemma}
\begin{proof}
Except for the calls to other algorithms, DevelChunk
(alg.~\ref{alg:dchunk}) has complexity $(\beta-\alpha)^2$, 
Get$\Gamma$ (alg.~\ref{alg:getg}) has complexity $d^2$, 
InverseChunk (alg.~\ref{alg:ichunk}) has complexity $(\beta-\alpha)^2$
and Residue (alg.~\ref{alg:residue}) has complexity $2d^2$.

Now the logarithmic binary decomposition of $\beta$ shows that
Get$\Gamma$, InverseChunk and Residue are each called less than
$\log_2\left(\frac{\beta+d}{2d}\right)$ times.
\end{proof}

\subsection{Fiduccia's algorithm for linear recurring sequences}
An alternative to the high-order lifting is to directly use 
Fiduccia's algorithm for
linear recurring sequences \citep{Fiduccia:1985:EFLR}. Its complexity
is slightly different and we show next that it can be interesting when
$\beta=\alpha$.

\subsubsection{Single coefficient recovery}
With the same notations as before,
one wants to solve $B = A \times T$ for $B$ and $A$ polynomials of degree
bounded by $d$ and $T$ a series development. We want to obtain
the coefficients of $T$ only between degrees $\alpha$ and $\beta$. 
The algorithm is as follows: solve directly for the first $d$ terms of
$T$ and afterwards, if $A = \sum_{i=0}^d a_i Z^i$, we obtain a recurring linear
sequence for the coefficients of $T=\sum^\infty t_i Z^i$:
\begin{equation}\label{eq:fiduc} a_0 t_\ell = - \sum_{i=1}^d a_i t_{\ell-i} \end{equation}

If $a_0 \neq 0$, let us define the characteristic polynomial $P(Z)=\text{rev}(A)/a_0=A(1/Z)Z^d/a_0$.
This induces the following linear system involving the companion
matrix of $P$, $C=\text{Companion}(P(Z))$:
\[
\left(
    \begin{array}{*{12}c}
     t_{\ell-d+1} \\
     \vdots\\
     t_{\ell}\\
     \end{array}
\right) =
\left(
    \begin{array}{*{12}c}
     0 & 1 & 0 & \ldots & 0 \\
     \vdots & & \ddots& \ddots & \vdots\\
     0 &  & \ldots &  0 & 1\\
     -\frac{a_d}{a_0} &  & \ldots & & -\frac{a_1}{a_0} \\
     \end{array}
\right) \times
\left(
    \begin{array}{*{12}c}
     t_{\ell-d} \\
     \vdots\\
     t_{\ell-1}\\
     \end{array}
\right) =
\text{C}^T \times
\left(
    \begin{array}{*{12}c}
     t_{\ell-d} \\
     \vdots\\
     t_{\ell-1}\\
     \end{array}
\right) 
\]
By developing the above system
$\ell-d$ times, we obtain one coefficient of $T$
with the simple dot product:
\begin{equation}
t_\ell = [0,\ldots,0,1] \times \left(C^T\right)^{\ell-d} \times
[t_1,\ldots,t_d]^T
= [t_1,\ldots,t_d] \times C^{\ell-d} \times [0,\ldots,0,1]^T
\end{equation}
The idea of Fiduccia is then to use the Cayley-Hamilton theorem, $P(C) = 0$, and
identify polynomials with a vector representation  of their
coefficients, in order to obtain
$C^{\ell-d} \times [0,\ldots,0,1]^T = Z^{\ell-d} \times Z^{d-1} = 
Z^{\ell-1} \mod P(Z)$.
Now the modular exponentiation is obtained by binary recursive
squaring in $\log_2(\ell-1)$ steps, each one involving 1 or 2 multiplications
and 1 or 2 divisions of degree $d$, depending on the bit pattern of
$\ell-1$.
Thus, the complexity of the modular exponentiation is bounded by $\log_2(\ell)(8d^2)$ with an average
value of $\log_2(\ell)(6d^2)$, exactly the same constant factor as for
the high-order lifting when $\beta=\alpha$. The additional
operations are just a dot product of the obtained polynomial with
$[t_1,\ldots,t_d]$ and the initial direct Taylor recovery of the latter
coefficients, thus yielding the overall complexity for a single
coefficient of the series of $\log_2(\ell)(6d^2)+d^2$ arithmetic operations.

\subsubsection{Cluster of coefficients recovery}
In the more generic case of several clustered coefficients, $\ell \in
[\alpha,\beta]$, one needs to modify the algorithm, in order to avoid
computing $\beta-\alpha$ products
by $[0,\ldots,0,1,0,\ldots,0]^T$. 
Instead one will recover $d$ coefficients at a time, in
$\frac{\beta-\alpha}{d}$ steps.

First
the binary recursive powering algorithm is used to get
$\frac{\beta-\alpha}{d}$ expressions 
of $C^{\alpha+(j-1)d} = \sum_{i=0}^{d-1} c_i^{(j)} C^i$, at an average
global arithmetic cost of
$\left(\frac{\beta-\alpha}{d}\log_2(d)+\log_2(\alpha)\right)\left(6d^2\right)$.
Then for $v=[t_1,\ldots,t_d]^T$, the sequence
$v,Cv,C^2v,\ldots,C^{d-1}v$ is computed once, iteratively. 
Finally
this sequence is combined with the coefficients $c_i^{(j)}$ to obtain
the $\beta-\alpha$ coefficients at an overall cost of 
$\left(\frac{\beta-\alpha}{d}\log_2(d)+\log_2(\alpha)\right)\left(6d^2\right)+4d^2+\max\{d^2;d(2(\beta-\alpha)-d\}$.

\subsubsection{High-Order and Fiduccia algorithm comparison}
We compare in Table~\ref{tab:compfidu} the arithmetic
complexity bound of Storjohann's high-order lifting and the average
complexity bound for Fiduccia's algorithm. 
\begin{table}[htb]\center\small\vspace{-15pt}
\begin{tabular}{|c|c||c|c|}
\hline
Algorithm &  $\ell=\beta=\alpha$ & $\ell \in [\alpha,\beta]$\\
\hline
\multirow{2}{*}{High-Order}& \multirow{2}{*}{$6d^2\log_2(\frac{\beta}{2d})+4d^2$} &
$\left(6d^2+2(\beta-\alpha)^2\right)\log_2(\frac{\beta}{2d})$
\\
& & \hfill$+d(\beta-\alpha+2d)$\\
\multirow{2}{*}{Fiduccia} & \multirow{2}{*}{$6d^2\log_2(\beta)+d^2$} &
$\left(6d^2\right)\left(\frac{\beta-\alpha}{d}\log_2(d)+\log_2(\alpha)\right)$\\
 & & \hfill$+d\left(2(\beta-\alpha)+3d\right)$\\
\hline
\end{tabular}
\caption{Complexities, for $\beta>2d$, of Storjohann's high-order lifting and
  Fiduccia's algorithm, the latter on average.}
 \label{tab:compfidu}
\end{table}

From this table we see that Storjohann's algorithm should be preferred
when $\beta \neq \alpha$ and that Fiduccia's algorithm should be
preferred when both conditions $\beta=\alpha$ and ``$d$ is small'' are
satisfied. 
In practice, on the bivariate 
examples of Section~\ref{sec:appli}, with $\beta=\alpha$, 
the differences remained within
$20\%$ and were always dominated by the fraction reconstruction. 
Therefore, in the following, we use preferably Storjohann's high-order
lifting.

\subsection{Bivariate lifting}

We come back to the bivariate case of Equation~(\ref{eq:mgfG}). $B$, $A$
and all the $g_l$ are polynomials in $y$ (not fractions).  Therefore
one can compute the lifting on $z$ using arithmetic modulo $y^{\enn+1}$
for the coefficients.  Operations in this latter domain thus costs no
more than $O(\enn^2)$ operations over $\Q$.
In the following we use the formalism of the high-order lifting, but
the algorithm of Fiduccia can be used as a replacement if needed.

Finally, for faster computations,
one can also convert the rational coefficients
to their floating point approximations in order to get
Algorithm~\ref{alg:bivlif}.
\begin{algorithm}[H]
\caption{Bivariate floating point lifting.}\label{alg:bivlif}
\begin{algorithmic}[1]
\REQUIRE $\frac{B(y,z)}{A(y,z)} \in \Q[y](z)$.
\ENSURE A floating point $\left[ [BA^{-1}]_\alpha^\beta \right]_0^\enn$.
\STATE $B_f(y,z)$ and $A_f(y,z)$ are the conversion of $B$ and $A$ in
floating points. 
\STATE $(\Gamma_i) = ${\bf High-Order}$(A_f,\alpha,\beta)$ modulo $y^{\enn+1}$;
\STATE $(g_j(y))_{j=\alpha,\ldots,\beta}= ${\bf DevelChunk}$(A_f,B_f,(\Gamma_i),\alpha,\beta)$ modulo $y^{\enn+1}$;
\STATE Return $([g_j(y)]_0^\enn)_{j=\alpha,\ldots,\beta}$.
\end{algorithmic}
\end{algorithm}

Then floating point arithmetic modulo $y^{\enn+1}$, together with
Lemmata~\ref{lem:chunk} and~\ref{lem:hol},
yield the following complexity for the computation of chunks of the
Taylor development of a bivariate polynomial fraction:
\begin{proposition}Let $G(y,z)=\frac{B(y,z)}{A(y,z)}$ be a rational
  fraction with $B$ and $A$ polynomials of degrees less than $d$ with
  floating point coefficients.  Suppose now that $\beta >> d$, and
  that $\beta-\alpha = O(d)$.  Then the overall complexity to compute
  $\left[ [BA^{-1}]_\alpha^\beta \right]_0^\enn$ with
  Algorithm~\ref{alg:bivlif} and classical polynomial arithmetic is
  $$O(\log(\beta)d^2\enn^2)$$
  rational operations.
\end{proposition}

This improves e.g. on the algorithm of \citet{Knuth:1997:TAoCPSA} used
in \citet[][Algorithm 8]{nicodeme}, which has
complexity $O(\log(\beta)d^3\enn^2)$.

\jgd{Note that, with fast floating point polynomial arithmetic (for
instance using FFT), our complexity would reduce to   
$$O(\log(\beta)d\log(d)\enn\log(\enn))=O((\enn d)^{1+o(1)}\log(\beta)).$$}

\subsection{Overall complexity}
Another improvement can be made on Algorithm~\ref{alg:rr} when the
desired degree $\enn$ in $y$ of the development is smaller than the
degree $d$ of the obtained bivariate fraction: compute the series and the
fraction reconstruction also modulo $y^{\enn+1}$. Recall that we
consider computing
$\mathbb{P}(N_\ell=n)$ where the transition matrix is of dimension
$R$, with $O(R)$ nonzero coefficients, and the rational fraction is
of degree $d\leq R$.
Therefore, the overall complexity of Algorithms~\ref{alg:rr}-\ref{alg:bivlif}, with fast
arithmetic, computing the latter, is bounded by:
\jgd{\begin{equation}
O\left(\min\{\enn,d\}d^2 R\log(R)+\log(\ell)d^{1+o(1)}\enn^{1+o(1)}\right)
\end{equation}
}

\section{Applications}\label{sec:appli}
All the transition matrices arising from the DFA of the considered
examples are available from the ``Sparse Integer Matrix
Collection''\footnote{\url{http://ljk.imag.fr/membres/Jean-Guillaume.Dumas/Matrices/DFA/}}.
\subsection{Toy-examples}
We consider here an independent and identically distributed sequence of
letters that are uniformly distributed over the four letter alphabet
$\mathcal{A}=\{{\tt A},{\tt B},{\tt C},{\tt D}\}$. Partial recursion
was performed with a floating point-arithmetic precision of
$\varepsilon=1/2^{1024} \simeq 10^{-710}$ (implementation using the 
${\tt mpf}$ class from the GMP\footnote{GNU Multi-Precision library \url{http://gmplib.org/}}), and relative error
$\eta=10^{-15}$. \gregtwo{Note that the high $1024 bit$ precision was necessary to avoid numerical stability issues in the numerical convergence of the partial recursion algorithm. These issues seem to be directly connected with the order of magnitude of $1-\lambda$ and the chosen precision ($1024 bits$) solved these issues for all the computation considered. It is nevertheless obvious that this phenomenon requires further investigation. See discussion.}

 The bivariate polynomial fraction reconstruction was
implemented using \linbox\footnote{\url{http://linalg.org}} and
\givaro\footnote{\url{http://ljk.imag.fr/CASYS/LOGICIELS/givaro}} and
the high-order lifting using \givaro~and
\mpfr\footnote{\url{http://www.mpfr.org}} with the {\sc
  mpreal}\footnote{\url{http://www.holoborodko.com/pavel/mpfr}}
C++ wrapper. 
\jgd{All running times are expressed in seconds and MT stands
for Memory Thrashing.}

\begin{table}
{\setlength{\arraycolsep}{4pt}
$$
  \begin{array}{|c|cc|lr|cr|}
    \hline
    \multicolumn{1}{|c|}{\text{POSIX regex}}       & \multicolumn{1}{c}{R} & \multicolumn{1}{c|}{F}  & \multicolumn{1}{c}{1-\lambda} & \multicolumn{1}{c|}{t_1} & \multicolumn{1}{c}{\text{frac. deg.}} & \multicolumn{1}{c|}{t_2} \\
    \hline
    {\tt AD(A|D)\{0\}AD} & 5 & 1 & 3.7 \times 10^{-3} & 0.03 & 2/4 & 0.00 \\
    {\tt AD(A|D)\{2\}AD} & 12 & 2 & 9.5 \times 10^{-4} & 0.11 & 6/8 & 0.01 \\
    {\tt AD(A|D)\{5\}AD} & 50 & 8 & 1.2 \times 10^{-4} & 0.49 & 28/30 & 0.12 \\
    {\tt AD(A|D)\{10\}AD} & 555 & 89 & 3.7 \times 10^{-5} & 6.14 & 321/323 & 3.18 \\
    {\tt AD(A|D)\{15\}AD} & 6155 & 987 & 1.2 \times 10^{-7} & 73.46 &
    3570/3572 & 17035.18 \\
   \hline
  \end{array}
$$
}
\caption{Toy-example motifs over the alphabet $\mathcal{A}=\{{\tt
    A},{\tt B},{\tt C},{\tt D}\}$. $R$ (resp. $F$) is the number of
  states (resp. final states) of the minimal order $0$ DFA associated
  to the regular expression. $\lambda$ is the largest eigenvalue of
  $\mathbf{P}$, and $t_1$ the time to compute $\lambda$ using the
  power method. ``frac. deg.'' corresponds to the fractional degrees
  of $G(y,z)$ and $t_2$ is the time to compute $G(y,z)$ using
  Algorithm~\ref{alg:rr}.}
 \label{tab:toy_motifs}
\end{table}

In Table~\ref{tab:toy_motifs} we consider 5 example motifs of
increasing complexities. For the partial recursion approach, the
eigenvalue $\lambda$ is reported along with the corresponding
computational time. 
\jgdeig{One should note that the high cost of this part is mainly
due to the 1024 bits precision floating point computation rather than
to the crudeness of the power approach (For computing the largest
eigenvalue, Lanczos iterations, the power approach, multigrid methods
etc. display similar performance on the considered examples: that is
less than 1 second for any of them in double precision. Our first
implementation in arbitrary precision used the power method, there 
of course being room for improvement).}

We also report in this
table the fractional degrees of $G(y,z)$ computed through the rational
reconstruction. We can see that the limiting factor
of the series computation is memory. For example, for Motif ${\tt
  AD(A|D)\{15\}AD}$, only storing the first $2d=3570+3572=7142$
bivariate terms over the rationals of the series would require the
order of $4d^3 R\log_{2^8}(R)\approx 1.7 \times 10^6$ Gigabytes,
using the estimates of Proposition~\ref{prop:fr}. Note that for this
motif, the degrees in $z$ of the numerator and denominator of the
fraction are only probabilistic since they were computed modulo a
random word-size prime number at a random evaluation in $y$.

\begin{table}\vspace{-10pt}
{\setlength{\arraycolsep}{2pt}
\begin{small}$$
\begin{array}{|c|rrrl|rr|rr|rr|}
\hline

& \multicolumn{1}{c}{n} & \multicolumn{1}{c}{\alpha} & \multicolumn{1}{c}{\ell} & \multicolumn{1}{c|}{\mathbb{P}(N_\ell=n)} & \multicolumn{1}{c}{\text{e.r.}} & \multicolumn{1}{c|}{t_0} & \multicolumn{1}{c}{t_3} & \multicolumn{1}{c|}{+t_1} & \multicolumn{1}{c}{t_4} & \multicolumn{1}{c|}{+t_2} \\
\hline
\multirow{6}{*}{\begin{sideways}{\tt AD(A|D)\{0\}AD}\end{sideways}} &  10 & 90 & 2,000 & 9.12559 \times 10^{-2} & 0.234 & 0.50 & 0.04 & 0.07 & 0.01 & 0.01 \\
                     &     &    & 20,000 & 4.37982 \times 10^{-21} &
                     0.168 & 5.00 & 0.04 & 0.07 & 0.01 & 0.01 \\
                     &     &    & 200,000 & 3.82435 \times 10^{-302} & 0.063 & 49.92 & 0.04 & 0.07 & 0.01 & 0.01 \\
                     &  100 & 666 & 2,000 & 9.06698 \times 10^{-59} & 0.025 & 4.47 & 2.53 & 2.56 & 0.01 & 0.01 \\
                     &      &     & 20,000 & 2.95125 \times 10^{-3} & 0.586 & 46.04 & 2.53 & 2.56 & 0.01 & 0.01 \\
                     &      &     & 200,000 & 1.07460 \times 10^{-196} & 1.495 & 461.61& 2.53 & 2.56 & 0.01 & 0.01 \\
\hline
\multirow{6}{*}{\begin{sideways}{\tt AD(A|D)\{2\}AD}\end{sideways}}  &  10 & 128 & 2,000 & 6.06131 \times 10^{-5} & 0.025 & 1.12 & 0.13 & 0.24 & 0.01 & 0.02 \\
                     &    &     & 20,000 & 8.13580 \times 10^{-3} & 0.114 & 11.38 & 0.13 & 0.24 & 0.01 & 0.02 \\
                     &    &     & 200,000 & 2.54950 \times 10^{-67} & 0.158 & 113.13 & 0.13 & 0.24 & 0.01 & 0.02 \\
                     &  100 & 971 & 2,000 & 4.58582 \times 10^{-94} & 0.027 & 10.44 & 8.97 & 9.08 & 0.01 & 0.02 \\
                     &      &     & 20,000 & 1.14066\times  10^{-34} & 0.260 &  107.05 & 8.97 & 9.08 & 0.01 & 0.03 \\
                     &      &     & 200,000 &  5.92396\times  10^{-14} &  0.232 &  1075.90 & 8.97 & 9.08 & 0.01 & 0.03 \\
\hline
\multirow{6}{*}{\begin{sideways}{\tt AD(A|D)\{5\}AD} \end{sideways}} & 2 & 158 & 2,000 & 2.59931 \times  10^{-2} & 0.031 &  1.23 &  0.07 & 0.56 & 0.00 & 0.13 \\
                     &   &     & 20,000 & 2.55206 \times  10^{-1} & 0.040 &  12.80 &  0.07 & 0.56 & 0.01 & 0.13 \\
                     &   &     & 200,000 & 1.35276 \times  10^{-8} & 0.041 &  124.76 & 0.07 & 0.56 & 0.01 & 0.13 \\
                     & 20 & 278 & 2,000 & 1.59351 \times  10^{-22} & 0.055 &  8.76 & 2.18 & 2.67 & 0.02 & 0.64 \\
                     &    &     & 20,000 & 3.79239 \times  10^{-11} & 0. 126 &  88.19 & 2.18 & 2.67 & 0.03 & 0.65 \\
                     &    &     & 200,000 & 5.79753\times  10^{-2} & 0.044 &  912.11 & 2.18 & 2.67 & 0.04 & 0.66 \\
\hline
\multirow{6}{*}{\begin{sideways}{\tt AD(A|D)\{10\}AD}\end{sideways}}  & 2 & 75 & 2,000 & 2.38948 \times  10^{-4} & 0.017 &  14.38 &  1.05 & 7.19 & 0.13 & 27.90 \\
                      &   &    & 20,000 & 4.4012 \times  10^{-3} & 0.093 &  148.49 &  1.05 & 7.19 & 0.32 & 28.07 \\
                      &   &    & 200,000 & 1.33166 \times  10^{-1} &
                      \text{NA} & \text{NA} &  1.05 & 7.19 & 0.48 &
                      28.12 \\
                      & 20 & 380 & 2,000 & 1.24717 \times  10^{-27} & 0.000 &  100.45 &  34.41 & 40.55 & 0.80 & 261.21 \\
                      &    &     & 20,000 &  1.25298 \times  10^{-25} & \text{NA} &  \text{NA} & 34.41 & 40.55 & 1.84 & 263.35 \\
                      &    &     & 200,000 &  6.25326 \times  10^{-18} & \text{NA} &  \text{NA} & 34.41 & 40.55 & 2.72 & 264.05 \\
\hline
\multirow{6}{*}{\begin{sideways}{\tt AD(A|D)\{15\}AD}\end{sideways}}  & 2 & 87 & 2,000 &  6.74582\times  10^{-6} & 0.001 &  153.54 &  12.95 & 86.41 & 0.16 & 17035.34 \\
                      &   &    & 20,000 &  7.02066\times  10^{-5} & \text{NA} &  \text{NA} &  12.95 & 86.41 & \text{-} & \text{-} \\
                      &   &    & 200,000 &  9.09232\times  10^{-4} & \text{NA} &  \text{NA} &  12.95 & 86.41 & \text{-} & \text{-} \\
                      & 20 & 491 & 2,000 &  5.72720\times  10^{-30} & \text{NA} &  \text{NA} &  477.05 & 550.51 & \text{-} & \text{-} \\
                      &    &     & 20,000 & 6.39056 \times  10^{-29} & \text{NA} &  \text{NA} &  477.05 & 550.51 & \text{-} & \text{-} \\
                      &    &     & 200,000 &  1.42666\times  10^{-27} & \text{NA} &  \text{NA} &  477.05 & 550.51 & \text{-} & \text{-} \\
\hline
\end{array}
$$\end{small}
}
\caption{$\mathbb{P}(N_\ell=n)$ for the toy-example motifs over the alphabet $\mathcal{A}=\{{\tt A},{\tt B},{\tt C},{\tt D}\}$ using a i.i.d. and uniformly distributed background model. $\alpha$ is the rank of the partial recursion (depends only on $n$), ``e.r.'' is the ratio of the relative error of the computation divided by the targeted relative error $\eta=10^{-15}$, $t_0$ is the running time to perform the computation using the full recursion, $t_3$ is the running time to perform the computation using the partial recursion (``$+t_1$'' gives the total running time $t_1+t_3$), $t_4$ is the running time to perform the computation using the high-order lifting (``$+t_2$'' give the total running time $t_2+t_4$).}\label{tab:toy_res}
\end{table}

In Table~\ref{tab:toy_res} we perform the computation of
$\mathbb{P}(N_\ell=n)$ for the considered motifs for various ranges
of values for $\ell$ and $n$. For validation purposes, the results of
the partial recursion are compared to those of the slower full
recursion. The relative error between the two approaches is compared
to expected relative error $\eta$: in all cases but one the resulting
error ratio (e.r.) is below 1.0 thus proving that both results are
quite consistent. In the remaining case, e.r. is only slightly larger
than 1.0 (1.495) which remains acceptable. In terms of computational
time however, the partial recursion approach is dramatically faster
than the full recursion. This is especially the case for the more
complex motifs for which full recursion was not even performed in some
cases. 

With the high-order lifting approach (Algorithms
\ref{alg:rr}-\ref{alg:bivlif}) we see that whenever the degree of the
bivariate fraction remains small, the overall performance is very
good. Moreover, one could compute the fraction once and then use the
very fast high-order lifting to recover any coefficient at a
negligible cost. Now when the degrees and the size of the involved
matrices grows, memory becomes the limiting factor, just to store the
series, prior to any computation on it.

In empirical complexity, the full recursion increases at the
expected $O(n \times \ell)$ rate. On the other hand, the partial
recursion running time is consistent with an $O(\alpha \times n)$ rate
with $\alpha$ increasing at a roughly linear rate with $n$.

\subsection{Transcription factors in Human Chromosome 10}

In this section, we consider the complete DNA ($\mathcal{A}=\{{\tt
  A},{\tt C},{\tt G},{\tt T}\}$) sequence
  
  of the Human Chromosome 10. 
In order to take into account the codon (DNA words of size 3)
structure of the sequence \greg{(which is known to play a key role in
  coding sequences)}, we adjust a homogeneous order 2 Markov model on
the data\footnote{\greg{Homogeneous order $m\geqslant 0$ Markov model
    MLE uses the observed frequencies of $(m+1)$-words. Taking into
    account the codons' ($3$-letter words) frequency hence requires to
    consider at least an order $m=2$ Markov model.}}. 
\gregtwo{The tri-nucleotide frequencies are given in
  Table~\ref{tab:tri}; sequence length is $\ell=131,624,728$ and the
  sequence starts with the two letters ${\tt  GA}$.} 
The maximum likelihood estimate (MLE) of the transition matrix of the model is
directly obtained from the observed counts of all DNA words of size
3. For example, since $N({\tt TAA})=2632852$, $N({\tt TAC})=1451956$,
$N({\tt TAG})=1655432$ and $N({\tt TAT})=2565811$, we get the MLE:

$$\widehat{\mathbb{P}}(X_i={\tt C} | X_{i-2}X_{i-1} = {\tt
  TA})=\frac{1451956}{2632852+1451956+1655432+2565811}.$$ One should
note that our Markov parameters are then all rationals.

\begin{table}[htb]
\setlength{\arraycolsep}{3pt}
\begin{footnotesize}\begin{multicols}{4}
\gregtwo{
$$
\begin{array}{|c|c|}
\hline
{\tt AAA}  &   4925908\\
{\tt AAC}  &   1894781\\
{\tt AAG}  &   2608606\\
{\tt AAT}  &   3178544\\
{\tt ACA}  &   2643624\\
{\tt ACC}  &   1556255\\
{\tt ACG}  &   346765 \\
{\tt ACT}  &   2095819\\
{\tt AGA}  &   2893502\\
{\tt AGC}  &   1890960\\
{\tt AGG}  &   2394790\\
{\tt AGT}  &   2097405\\
{\tt ATA}  &   2563310\\
{\tt ATC}  &   1743008\\
{\tt ATG}  &   2414219\\
{\tt ATT}  &   3176591\\
\hline
\end{array}
$$
}

\gregtwo{
$$
\begin{array}{|c|c|}
\hline
{\tt CAA}  &   2467088\\
{\tt CAC}  &   2029304\\
{\tt CAG}  &   2749113\\
{\tt CAT}  &   2412170\\
{\tt CCA}  &   2494477\\
{\tt CCC}  &   1800468\\
{\tt CCG}  &   378469 \\
{\tt CCT}  &   2400823\\
{\tt CGA}  &   303297 \\
{\tt CGC}  &   325579 \\
{\tt CGG}  &   377677 \\
{\tt CGT}  &   347404 \\
{\tt CTA}  &   1654612\\
{\tt CTC}  &   2267140\\
{\tt CTG}  &   2751827\\
{\tt CTT}  &   2613317\\
\hline
\end{array}
$$
}

\gregtwo{
$$
\begin{array}{|c|c|}
\hline
{\tt GAA}  &   2581990\\
{\tt GAC}  &   1266423\\
{\tt GAG}  &   2263506\\
{\tt GAT}  &   1740603\\
{\tt GCA}  &   1946916\\
{\tt GCC}  &   1632319\\
{\tt GCG}  &   323812 \\
{\tt GCT}  &   1891590\\
{\tt GGA}  &   2085526\\
{\tt GGC}  &   1630198\\
{\tt GGG}  &   1794047\\
{\tt GGT}  &   1554034\\
{\tt GTA}  &   1452165\\
{\tt GTC}  &   1267140\\
{\tt GTG}  &   2031739\\
{\tt GTT}  &   1900980\\
\hline
\end{array}
$$
}

\gregtwo{
$$
\begin{array}{|c|c|}
\hline
{\tt TAA}  &   2632852\\
{\tt TAC}  &   1451956\\
{\tt TAG}  &   1655432\\
{\tt TAT}  &   2565811\\
{\tt TCA}  &   2572660\\
{\tt TCC}  &   2085193\\
{\tt TCG}  &   304911 \\
{\tt TCT}  &   2898664\\
{\tt TGA}  &   2570179\\
{\tt TGC}  &   1947900\\
{\tt TGG}  &   2497293\\
{\tt TGT}  &   2653181\\
{\tt TTA}  &   2635963\\
{\tt TTC}  &   2584156\\
{\tt TTG}  &   2470768\\
{\tt TTT}  &   4937918\\
\hline
\end{array}
$$
}
\end{multicols}\end{footnotesize}
\caption{\gregtwo{Tri-nucleotide frequencies in the human chromosome 10 ($\ell=131\,624\,728$).}}\label{tab:tri}
\end{table}

\begin{table}[htb]
{\setlength{\arraycolsep}{3pt}
\begin{footnotesize}$$
  \begin{array}{|c|cc|lr|cr|}
    \hline
    \multicolumn{1}{|c|}{\text{Transcription Factor}}       & \multicolumn{1}{c}{R} & \multicolumn{1}{c|}{F}  & \multicolumn{1}{c}{1-\lambda} & \multicolumn{1}{c|}{t_1} & \multicolumn{1}{c}{\text{frac. deg.}} & \multicolumn{1}{c|}{t_2} \\
    \hline
    {\tt CGCACCC} & 21 & 1 &  1.04130 \times 10^{-5} & 0.13 & 18/19 & 3.24 \\
    {\tt TCCGTGGA} & 22 & 1 & 5.06531 \times 10^{-6} & 0.13 & 19/20 & 3.62 \\
    {\tt ACAACAAC} & 23 & 1 & 1.16022 \times 10^{-5} & 0.15 & 21/22 & 5.71 \\
    {\tt (A|C)TAAA(C|T)AA} & 25 & 2 & 1.41728 \times 10^{-4} & 0.18 & 20/20 & 4.36 \\
    {\tt (A|T)\{3\}TTTGCTC(A|G)} & 30 & 2 & 1.05501 \times 10^{-5} & 0.20 & 23/23 & 5.50 \\
    {\tt A\{24\}} &  38 & 1 & 6.11979 \times 10^{-11} & 0.25 & 36/37 & 3.78 \\
    {\tt TA(A|T)\{4\}TAG(A|C)} & 54 & 2  & 6.87736 \times 10^{-5} & 0.45 & 21/22 & 1.41 \\
    {\tt (C|T)CCN(C|T)TN(A|G)\{2\}CCGN} & 66 & 4  & 3.21470 \times 10^{-6} & 0.63 & 24/25 & 9.57 \\
    {\tt GCGCN\{6\}GCGC} & 228 & 8 & 3.49649 \times 10^{-7} & 3.84 & 54/55 & 66.52 \\
    {\tt CGGN\{8\}CGG} & 419 & 13 & 8.20997 \times 10^{-6} & 10.12 & 81/82 & 283.61\\
    {\tt TTGACAN\{17\}TATAAT} &  2068 & 34 & 1.29222 \times 10^{-7} & 34.91 & 173/173 & 6392.23 \\
    {\tt TTGACAN\{16,18\}ATATAAT} &  2904 & 55 & 1.19636 \times 10^{-7} & 49.18 & 253/253 & 23727.28\\
    {\tt GCGCN\{15\}GCGC} &  6158 & 225 & 3.49683 \times 10^{-7} & 202.48 & 1079/1080 & \text{MT} \\
    \hline
  \end{array}
$$\end{footnotesize}
}
\caption{Regular expression of Transcription Factors (TFs) defined
  over the DNA alphabet $\mathcal{A}=\{{\tt A},{\tt C},{\tt G},{\tt
    T}\}$  using the IUPAC notation ${\tt N}={\tt (A|C|G|T)}$. $R$ (resp. $F$) is the number of states (resp. final states) of the minimal order $2$ DFA associated to the TFs. $\lambda$ is the largest eigenvalue of $\mathbf{P}$, and $t_1$ the time to compute $\lambda$ using the power method. ``frac. deg.'' corresponds to the fractional degrees of $G(y,z)$, and $t_2$ is the time to compute $G(y,z)$ using Algorithm~\ref{alg:rr}.}
 \label{tab:TF_motifs}
\end{table}

In Table~\ref{tab:TF_motifs} we consider a selection of various
Transcription Factors (TFs) motifs. These TFs are typically involved
in the regulation of gene expression. The selected motifs range from
simple patterns (e.g. ${\tt CGCACCC}$) to highly complex ones (e.g.
${\tt GCGCN\{15\}GCGC}$). For each motif, the precomputations
necessary for the partial recursion (computation of $\lambda$) and the
high-order lifting approach (computation of $G(y,z)$) are
indicated. As in Table~\ref{tab:toy_motifs} we see that the running
time increases with the motif complexity, eventually resulting in a
memory thrashing (MT) for the computation of the rational
reconstruction. One should note that time $t_2$ is larger for these
motifs than for the toy examples of the previous section even when the
fractional degrees are similar. This is explained by the more complex
nature of the model parameters (e.g. $1451956/8306051$ for the TFs {\it
  vs} $1/4$ for the toy-example).

\begin{table}\vspace{-10pt}
{\setlength{\arraycolsep}{3pt}
\begin{footnotesize}$$
\begin{array}{|c|rrl|rr|rr|}
\hline
\multicolumn{1}{|c|}{\text{Transcription Factor}} & \multicolumn{1}{c}{n} & \multicolumn{1}{c}{\alpha} & \multicolumn{1}{c|}{\mathbb{P}(N_\ell=n)} & \multicolumn{1}{c}{t_3} & \multicolumn{1}{c|}{+t_1} & \multicolumn{1}{c}{t_4} & \multicolumn{1}{c|}{+t_2} \\
\hline
{\tt CGCACCC} & 10 & 117 & 3.64365 \times 10^{-571} & 0.19 & 0.32 & 0.41 & 3.65 \\ 
              & 20 & 204 & 1.27159 \times 10^{-551} & 0.60 & 0.73 & 1.05 & 4.32 \\
              & 40 & 373 & 2.07574 \times 10^{-518} & 2.10 & 2.23 & 3.17 & 6.44 \\
\hline
{\tt TCCGTGGA} & 10 & 131 & 1.33747 \times 10^{-268} & 0.20 & 0.33 & 0.01 & 3.63 \\ 
               & 20 & 225 & 3.46367 \times 10^{-252} & 0.63 & 0.76 & 0.03 & 3.65 \\ 
               & 40 & 409 & 3.11336 \times 10^{-225} & 2.17 & 2.30 & 0.05 & 3.70 \\
\hline
{\tt AACAACAAC} & 10 & 142 & 3.86490 \times 10^{-170} &  0.25 & 0.40 & 0.02 & 5.73 \\
                & 20 & 258 & 1.22856 \times 10^{-155} & 0.88 & 1.03 & 0.03 & 5.76 \\
                & 40 & 492 & 1.69964 \times 10^{-132} & 3.24 & 3.39 & 0.06 & 5.79 \\
\hline
{\tt (A|C)TAAA(C|T)AA} & 10 & 136 & 6.76399 \times 10^{-8067} & 0.26 & 0.44 & 0.53 & 4.89 \\
                       & 20 & 240 & 4.79070 \times 10^{-8036} & 0.87 &
                       1.05 & 1.35 & 5.71 \\
                       & 40 & 449 & 3.22178 \times 10^{-7980} & 3.14 & 3.32 & 4.07 & 8.44 \\
\hline
{\tt (A|T)\{3\}TTTGCTC(A|G)} & 10 & 150 & 6.03263 \times 10^{-579} &
0.30 & 0.50 & 0.63& 6.13 \\
                             & 20 & 267 & 2.40165 \times 10^{-559} & 0.99 & 1.19 & 1.59 & 7.12 \\
                             & 40 & 500 & 5.10153 \times 10^{-526} & 3.58 & 3.78 & 4.87 & 10.42 \\
\hline
{\tt A\{24\}} &  5 & 171 & 1.16314 \times 10^{-4} & 0.31 & 0.56 & 0.02 & 8.73 \\
              & 10 & 310 & 1.09217 \times 10^{-6} & 1.01 & 1.26 & 0.02 & 17.75 \\
              & 20 & 589 & 9.62071 \times 10^{-11} & 3.68 & 3.93 & 0.04 & 40.27 \\
\hline
{\tt TA(A|T)\{4\}TAG(A|C)} &  5 & 93 & 1.60427 \times 10^{-3914} & 0.21 & 0.66 & 0.12 & 3.20 \\
                           & 10 & 148 & 3.23597 \times 10^{-3899} & 0.54 & 0.99 & 0.27 & 5.49 \\
                           & 20 & 256 & 1.79579 \times 10^{-3871} & 1.69 & 2.14 & 0.76 & 7.40 \\
\hline
{\tt (C|T)CCN(C|T)TN(A|G)\{2\}CCGN} &  5 & 150 & 1.94195 \times 10^{-173} & 0.43 & 1.06 & 0.02 & 9.59 \\
                                    & 10 & 215 & 8.71218 \times 10^{-165} & 1.00 & 1.63 & 0.02 & 8.60 \\
                                    & 20 & 342 & 2.39167 \times 10^{-150} & 3.01 & 3.64 & 0.05 & 8.63 \\
\hline
{\tt GCGCN\{6\}GCGC} & 1 & 65 & 4.73516 \times 10^{-19} & 0.24 & 4.08 & 0.62 & 67.09 \\
                     & 2 & 92 & 1.08880 \times 10^{-17} & 0.50 & 4.34 & 0.87 & 98.62 \\
                     & 4 & 138 & 1.91912 \times 10^{-15} & 1.22 & 5.06 & 1.49 & 155.99 \\
\hline

{\tt CGGN\{8\}CGG} & 1 & 82 & 5.21188 \times 10^{-467} & 0.59 & 10.71 & 1.47 & 284.93 \\
                   & 2 & 114 & 2.80818 \times 10^{-464} & 1.17 & 11.29 & 1.79 & 403.75 \\
                   & 4 & 169 & 2.71751 \times 10^{-459} & 2.74 & 12.86 & 3.12 & 651.20 \\

\hline
{\tt TTGACAN\{17\}TATAAT} & 1 & 92 & 6.97988 \times 10^{-7} & 3.06 & 37.97 & 2.49 & 6394.73 \\
                          & 2 & 137& 5.93598 \times 10^{-6} & 6.72 & 41.63 & 3.78 & 9378.03\\
                          & 4 & 199& 1.43106 \times 10^{-4} &15.23 & 50.14 & 6.63 & 15008.84 \\
\hline
{\tt TTGACAN\{16,18\}ATATAAT} & 1 & 96 & 2.28201 \times 10^{-6} & 4.86 & 54.04 & 5.30 & 23732.58 \\
                              & 2 & 129& 1.79676 \times 10^{-5} & 9.38 & 58.56 & 7.42 & 34949.45 \\
                              & 4 & 202& 3.71288 \times 10^{-4} &23.16 & 72.34 & 15.65 & 56832.04 \\
\hline
{\tt GCGCN\{15\}GCGC} & 1 & 119 & 4.71467 \times 10^{-19} & 12.62 & 215.10 & \text{-} & \text{MT} \\
                    & 2 & 173 & 1.08420 \times 10^{-17} & 27.15 & 229.63 & \text{-} & \text{MT} \\
                    & 4 & 255 & 1.91136 \times 10^{-15} & 63.45 & 265.93 & \text{-} & \text{MT} \\
\hline
\end{array}
$$\end{footnotesize}
}\vspace{-10pt}
\caption{$\mathbb{P}(N_\ell=n)$, \greg{with $\ell=131\,624\,728$}, for the TFs motifs over the alphabet $\mathcal{A}=\{{\tt A},{\tt C},{\tt G},{\tt T}\}$ using an order 2 homogeneous Markov model. $\alpha$ is the rank of the partial recursion (depends only on $n$), $t_3$ is the running time to perform the computation using the partial recursion (``$+t_1$'' gives the total running time $t_1+t_3$), $t_4$ is the running time to perform the computation using the high-order lifting (``$+t_2$'' give the total running time $t_2+t_4$).}\label{tab:TF_res}
\end{table}

In Table~\ref{tab:TF_res}, we can see the computed values of
$\mathbb{P}(N_\ell=n)$ for our TFs motifs and for various values of
$n$. Due to the large value of $\ell$, the full recursion was no longer
tractable and there is hence no reference value for the
probability of interest. However, the results of both approaches are
always the same (up to the requested relative precision). For low
complexity TFs, the high-lifting is always much faster than the
partial recursion when considering only the core
computations. However, we get the opposite results when we consider as
well the pre-computation time (i.e.: obtaining $G(y,z)$ for the
high-order lifting, or computing $\lambda$ for the partial
recursion). As for the toy-examples, we see that the high-order
lifting approach cannot cope with high complexity patterns since the
fractional reconstruction is not feasible for them.

\subsection{Protein signatures}

We consider here the complete human proteome build as the concatenation
of all human protein sequences over the 20 aminoacid alphabet (from
the Uniprot database\footnote{\url{http://www.uniprot.org}}) resulting in a
unique sequence of length $\ell=9,884,385$. We fit an order 0 Markov
model onto these data from the observed counts of all aminoacids:
{\setlength{\arraycolsep}{4pt}
\begin{footnotesize}$$
\begin{array}{lllll}
  N({\tt A})=691113, & N({\tt R})=555875, & N({\tt N})=357955, & N({\tt D})=472303, & N({\tt C})=227722, \\
  N({\tt E})=698841, & N({\tt Q})=469260, & N({\tt G})=649800, & N({\tt H})=258779, & N({\tt I})=432849, \\
  N({\tt L})=981769, & N({\tt K})=567289, & N({\tt M})=212203, & N({\tt F})=363883, & N({\tt P})=617242, \\
  N({\tt S})=816977, & N({\tt T})=529157, & N({\tt W})=119979, & N({\tt Y})=267663, & N({\tt V})=593726. \\
\end{array}
$$\end{footnotesize}
}
As a consequence, our MLE parameters are expressed as
rationals. For example: $\widehat{\mathbb P}(X_i={\tt
  W})=119979/9884385$.

\begin{table}\small
{\setlength{\arraycolsep}{4pt}
$$
  \begin{array}{|l|c|r|r|lr|cr|}
    \hline
    \multicolumn{1}{|c|}{\text{PROSITE signature}}      & \multicolumn{1}{c|}{\text{AC}} & \multicolumn{1}{c|}{R} & \multicolumn{1}{c|}{F} & \multicolumn{1}{c}{1-\lambda} & \multicolumn{1}{c|}{t_1} & \multicolumn{1}{c}{\text{degrees}} & \multicolumn{1}{c|}{t_2} \\
    \hline
    {\tt PILI\_CHAPERONE} & {\tt PS00635}& 46 & 1 & 1.7 \times 10^{-10} & 0.63 & 15/18 & 0.27 \\
    {\tt EFACTOR\_GTP}     & {\tt PS00301} & 52 & 4 & 1.0 \times 10^{-8} & 0.74 & 14/16 & 0.18 \\
    {\tt ALDEHYDE\_DEHYDR\_CYS} & {\tt PS00070} & 67 & 17 & 1.1 \times 10^{-6} & 0.91 & 11/12 & 0.21\\
    {\tt SIGMA54\_INTERACT\_2} & {\tt PS00676} & 85 & 1 & 8.8 \times 10^{-10} & 1.08 & 0/16 & 0.22\\
    {\tt ADH\_ZINC}& {\tt PS00059} & 87 & 8 &  2.2 \times 10^{-7} & 1.40 & 37/41 & 2.07 \\
    {\tt SUGAR\_TRANSPORT\_1}& {\tt PS00216} & 188 & 54 & 6.7 \times 10^{-7} & 3.48 & 17/18 & 1.05\\
    {\tt THIOLASE\_1}& {\tt PS00098}& 254 & 6 & 2.6 \times 10^{-15} & 5.21 &  37/38 & 1.76\\
    {\tt FGGY\_KINASES\_2}& {\tt PS00445}& 463 & 6 & 2.2 \times 10^{-7} & 11.52 & 26/30 & 2.39 \\
    {\tt PTS\_EIIA\_TYPE\_2\_HIS}&{\tt PS00372}& 756 & 46 & 1.3 \times 10^{-9} & 17.47 & 77/80 & 18.60\\
    {\tt SUGAR\_TRANSPORT\_2}&{\tt PS00217}& 1152 & 40 & 8.6 \times 10^{-7} & 36.95 & 149/151 & 68.36 \\
    \hline
  \end{array}
$$
}
\caption{Characteristics of some PROSITE signatures defined over the aminoacid alphabet. AC is the accession number in the PROSITE database, $R$ (resp. $F$) is the number of states (resp. final states) of the minimal order $2$ DFA associated to the signatures. $\lambda$ is the largest eigenvalue of $\mathbf{P}$, and $t_1$ the time to compute $\lambda$ using the power method. ``frac. deg.'' corresponds to the fractional degrees of $G(y,z)$, and $t_2$ is the time to compute $G(y,z)$ using Algorithm~\ref{alg:rr}.}
 \label{tab:PROSITE_motifs}
\end{table}

\begin{table}\small
{\setlength{\arraycolsep}{4pt}
$$
\begin{array}{|l|rrl|rr|rr|}
\hline
\multicolumn{1}{|c|}{\text{PROSITE signature}} & \multicolumn{1}{c}{n} & \multicolumn{1}{c}{\alpha} & \multicolumn{1}{c|}{\mathbb{P}(N_\ell=n)} & \multicolumn{1}{c}{t_3} & \multicolumn{1}{c|}{+t_1} & \multicolumn{1}{c}{t_4} & \multicolumn{1}{c|}{+t_2} \\
\hline
{\tt PILI\_CHAPERONE} &  5 & 125 & 1.20635 \times 10^{-16} & 0.98 & 1.16 &0.01  & 0.17 \\
                     & 10 & 221 & 5.78491 \times 10^{-35} & 3.17 & 3.80 & 0.01  & 0.28 \\
                     & 20 & 413 & 1.81452 \times 10^{-74} & 11.20 & 11.83 & 0.01  & 0.69\\
\hline
{\tt EFACTOR\_GTP} & 5 & 114 & 7.25588 \times 10^{-8} &  1.07 & 0.84 & 0.01 & 0.18\\
                   & 10 & 196 & 2.30705 \times 10^{-17} &  3.41 & 0.95 & 0.01&0.32\\
                   & 20 & 364 & 3.18090 \times 10^{-39}  & 12.23 & 1.28 & 0.01 &0.75\\
\hline
{\tt ALDEHYDE\_DEHYDR\_CYS} & 5 & 88 & 1.85592 \times 10^{-2} & 1.13 & 2.04 & 0.01  & 0.20\\
                            & 10 & 151 & 1.15181 \times 10^{-1} & 3.59 & 4.50 & 0.01  & 0.30\\
                            & 20 & 270 & 6.05053 \times 10^{-3} &  12.27 & 13.18 & 0.01 & 0.60\\
\hline
{\tt SIGMA54\_INTERACT\_2} & 5 & 106 & 4.09432 \times 10^{-13} &  1.45& 2.53& 0.01 &0.23\\
                           & 10 & 189 & 6.70971 \times 10^{-28} & 4.68 & 5.76& 0.01&0.29\\
                           & 20 & 350 & 2.45724 \times 10^{-60} & 16.17& 17.25& 0.01 &0.40\\
\hline
{\tt ADH\_ZINC} &  5 & 116 & 4.51132 \times 10^{-2} & 1.91& 3.31& 0.01  & 2.08\\
                & 10 & 196 & 6.99469 \times 10^{-5} & 5.98& 7.38& 0.02  & 3.80\\
                & 20 & 352 & 2.29397 \times 10^{-13}& 20.52& 21.92& 0.04  & 7.57\\
\hline
{\tt SUGAR\_TRANSPORT\_1} & 1 & 60 & 6.06925 \times 10^{-3} & 0.65 & 4.13 & 0.05 & 1.10\\
                          & 2 & 75 & 1.64759 \times 10^{-2} & 1.29 & 4.77 & 0.08 & 1.51\\
                          & 4 & 110 & 5.41084 \times 10^{-2} & 3.30 & 6.78 & 0.12  & 2.29\\
\hline
{\tt THIOLASE\_1} & 1 & 85 & 2.54343 \times 10^{-8} & 1.23 & 6.44 & 0.01 & 1.77\\
                  & 2 & 127 & 1.61364 \times 10^{-15} & 2.85 & 8.06 &
                  0.01  & 2.04 \\
                  & 4 & 207 & 6.25151 \times 10^{-30} & 7.96 & 13.17 & 0.01  & 2.38\\
\hline
{\tt FGGY\_KINASES\_2} & 1 & 73 & 2.43018 \times 10^{-1} & 2.08 & 13.60 & 0.01  & 2.40\\
                       & 2 & 97 & 2.68005 \times 10^{-1} & 4.21 & 15.73 & 0.01  & 3.32\\
                       & 4 & 142 & 1.08650 \times 10^{-1} & 10.27 & 19.48 & 0.01  & 4.54\\
\hline
{\tt PTS\_EIIA\_TYPE\_2\_HIS} & 1 & 76 & 1.23843 \times 10^{-2} & 3.41 & 20.88 & 0.41  & 19.02\\
                              & 2 & 105 & 7.76535 \times 10^{-5} &
                              7.32 & 24.79 & 0.61 & 27.96 \\
                              & 4 & 163 & 1.0177 \times 10^{-9} & 19.18 & 36.65 & 1.08 & 44.72\\
\hline
{\tt SUGAR\_TRANSPORT\_2} & 1 & 96 & 1.71124 \times 10^{-3} & 6.78 & 43.73 & 1.37  & 69.73\\
                          & 2 & 124 & 7.28305 \times 10^{-3} & 13.51 & 50.46& 2.01  & 103.56\\
                          & 4 & 177 & 4.39742 \times 10^{-2} & 32.81 &
                          69.76 & 3.57  & 172.62 \\
\hline
\end{array}
$$
}
\caption{$\mathbb{P}(N_\ell=n)$, \greg{with $\ell=9\,884\,385$}, for the PROSITE signatures over the aminoacid alphabet an order 0 homogeneous Markov model. $\alpha$ is the rank of the partial recursion (depends only on $n$), $t_3$ is the running time to perform the computation using the partial recursion (``$+t_1$'' gives the total running time $t_1+t_3$), $t_4$ is the running time to perform the computation using the high-order lifting (``$+t_2$'' give the total running time $t_2+t_4$).}\label{tab:PROSITE_res}
\end{table}

We also consider a selection of 10
PROSITE\footnote{\url{http://expasy.org/prosite/}} signatures which
correspond to known functional motifs in proteins. In
Table~\ref{tab:PROSITE_motifs}, the complexity of the considered
motifs are studied along with the computational time to obtain
$\lambda$ (time $t_1$) or to obtain $G(y,z)$ (time $t_2$). Motifs are sorted by
increasing complexities, from Signature ${\tt PILI\_CHAPERONE}$ (whose
minimal DFA has $R=46$ states including $F=1$ final state) to
Signature ${\tt SUGAR\_TRANSPORT\_2}$ (whose minimal DFA has $R=1152$
states including $F=40$ final states). For both approaches, the
running time for the precomputations is similar but, as for
previous applications, we observe a steeper increase for the
fractional reconstruction when considering high complexity motifs.

In Table~\ref{tab:PROSITE_res} we compute $\mathbb{P}(N_\ell=n)$ for
all considered PROSITE signatures and a range of values for $n$. For
each combination, both the partial recursion and the high-order
lifting are performed and the two methods agree perfectly in their
results. As for the TFs, the fast Taylor expansion (time $t_4$) is
much faster than the recursion part (time $t_3$) but the
precomputation of $G(y,z)$ (time $t_4$) has a high cost, especially
for the signatures of high complexity, which is consistent with
previous observations.

\section{Conclusion}\label{sec:conclusion}

We have developed two efficient approaches to obtain the exact distribution
of a pattern in a long sequence. 
\jgd{Table~\ref{tab:comp} recalls the
different obtained complexities.}

\begin{table}[htbp]\center\footnotesize \ADDITIONCOLOR
\begin{tabular}{|l|l||l|l|}
\hline
& Approach  & Memory & Time\\
\hline
\multirow{6}{*}{\begin{sideways}{\hspace{10pt}Approxim.}\end{sideways}}&&&\\[-7pt]
&Full recursion& $O\left(n R\right)$ & $O\left(n |\mathcal{A}| R \ell \right)$\\[5pt]&&&\\[-10pt]
&\cite{ribeca} & $O\left(nR^2\log(\ell)\right)$&
$O\left(n^2R^3\log(\ell)\right)$\\[5pt]&&&\\[-10pt]
&Partial recursion  & $O\left(n R\right)$ & 
$O\left( n^4 |\mathcal{A}| R\log(\ell)/\log(\nu^{-1})\right)$\\[2pt]
\hline 
\multirow{4}{*}{\begin{sideways}{\hspace{5pt}Exact}\end{sideways}}&&&\\[-7pt]
&\cite{nicodeme} & $O\left(n^3R^4\log(R)\right)$ &
$O\left(R^6\log(R)+n^5\log(\ell)\right)$ \\[5pt]&&&\\[-10pt]
&High-order & $O\left(nd^2R\log(R)\right)$ &
$O\left(nd^2R\log(R)+(nd)^{1+o(1)}\log(\ell)\right)$ \\[2pt]
\hline
\end{tabular}
\caption{Complexities of the different approaches to compute
  $\mathbb{P}(N_\ell = n)$. $R$ is the size of the automaton,
  $|\mathcal{A}|$ is the size of the alphabet, $d<R$ is the degree
  of the rational fraction and $0<\nu<1$ is the magnitude of the second
  largest eigenvalue of $\widetilde{\mathbf{P}}$.}\label{tab:comp}
\end{table}

The first approach uses a partial recursion and is suitable even
for high complexity patterns. Unfortunately, its quadratic complexity
in the observed number of occurrence $n$ makes it not recommended for
large values of~$n$.

The second approach has two steps: first obtaining $G(y,z)$ through
fraction reconstruction of the series $\sum G_\ell(y) z^\ell$ and then
performing fast Taylor expansion using high-order lifting. On the one
hand, in all examples, just computing the series appears to be the
bottleneck, especially for high complexity patterns. On the other
hand, the fast Taylor expansion is usually very fast, even if the
running time increases with the fractional degrees. 
\jgd{Moreover, once the
generating function has been obtained, the fast liftings can reveal the
distribution of {\em any} pattern count at a very low cost.}

Future work will include improvement of the precomputation of
$G(y,z)$, for instance reconstructing the rational fraction from
approximated evaluations of the series. 
\jgd{However, an exact reconstruction on approximate values could yield a
reasonable model, but only with a generic behavior. That is, $d$, the
obtained degree, would in general be equal to $R$, the size of the
input matrices. On the contrary, in the examples, this degree is much
lower in practice.}
\greg{One should also note that the distribution of the number of
  motif occurrences is known to be very sensitive to the parameters
  (the transition probabilities of the Markov chain) and that any
  approximation performed on these parameter might have large and
  uncontrolled consequences \citep{nuel2006pattern}.}

\jgd{Another solution could be to use regularization methods or the
  approximate gcd-like approaches of e.g. \cite{Kaltofen:2007:snc} for
  the pre-computation of $G(y,z)$. This could yield significant
  improvements both in terms of memory and computational time if the
  small degrees were preserved.}

\gregtwo{Concerning the partial recursion, it is clear that the need of high precision floating point computations to avoid numerical convergence instability is a major issue and that a careful investigation of the overall stability of the algorithm in floating point arithmetic is a top priority for further investigations.}

Overall, the high-order lifting approach is very efficient for low or
median complexity motifs, but cannot deal efficiently with the highly
complex motifs.  In our examples, we dealt with two real applications
(TFs in Human Chromosome 10, and PROSITE signatures in the Human complete
proteome) which demonstrate the practical interest of our approaches.
Finally, dealing with fully exact computations for frequent (large
$n$) and high complexity (large $R$) motifs yet remains an open
problem. \greg{At the present time, for such challenging problems, it
  is likely that one can only rely on approximations like the Edgeworth
  expansions or the Bahadur-Rao approximations \citep[see][for more
  details]{nuel2011pattern}.}

\bibliographystyle{plainnat}   
\bibliography{biblio}   

\end{document}